\documentclass [article] {amsart}
\usepackage{amsfonts,amsthm,amsbsy,amsmath,amssymb}
\usepackage{latexsym,enumerate,enumitem}
\usepackage{tikz-cd}
\usetikzlibrary{positioning}
\usepackage{xcolor}
\usepackage{mathrsfs}
\usepackage{mathabx}

\numberwithin{equation}{section}

\newtheorem{theorem}{Theorem}[section]
\newtheorem{corollary}[theorem]{Corollary}
\newtheorem{lemma}[theorem]{Lemma}
\newtheorem{proposition}[theorem]{Proposition}

\theoremstyle{definition}
\newtheorem{remark}[theorem]{Remark}

\newtheorem*{theorem*}{Theorem}
\newtheorem*{lemma*}{Lemma}
\newtheorem*{prop*}{Proposition}
\newtheorem*{corollary*}{Corollary}

\def\wh{\widehat}

\usepackage{hyperref}
\hypersetup{colorlinks, breaklinks,
	linkcolor = blue,
	urlcolor = blue,
	anchorcolor = blue,
	citecolor = blue}

\usepackage{cleveref}

\def\N{{\mathbb N}}

\def\R{{\mathbb R}}
\def\Z{{\mathbb Z}}

\def\S{\mathbb{S}}

\def\a{\alpha}
\def\b{\beta}

\def\d{\delta}

\def\h{\eta}

\DeclareFontFamily{U}{mathx}{\hyphenchar\font45}
\DeclareFontShape{U}{mathx}{m}{n}{
	<5> <6> <7> <8> <9> <10>
	<10.95> <12> <14.4> <17.28> <20.74> <24.88>
	mathx10
}{}
\DeclareSymbolFont{mathx}{U}{mathx}{m}{n}
\DeclareFontSubstitution{U}{mathx}{m}{n}
\DeclareMathAccent{\widecheck}{0}{mathx}{"71}
\DeclareMathAccent{\wideparen}{0}{mathx}{"75}

\begin{document} 

\title{Rough averages of triangular Hilbert transforms}

\author{Fred Yu-Hsiang Lin}
\address{Department of Mathematics, Purdue University, Mathematical Sciences Bldg, 150 N University St, West Lafayette, IN 47907 USA}
\email{lin2311@purdue.edu}

\author{Lenka~Slav{\'i}kov{\'a}}
\address{Department of Mathematical Analysis, Faculty of Mathematics and Physics, Charles University, Sokolovsk\'a 83, 186 75 Praha 8, Czech Republic}
\email{slavikova@karlin.mff.cuni.cz}


\begin{abstract}
We study a bilinear singular integral operator obtained by taking rough averages of certain directional variants of the triangular Hilbert transform. This operator can be interpreted as the twisted paraproduct with a rough homogeneous kernel. Under a balance condition on the $L^q$ integrability of the kernel on the unit sphere, we establish a range of $L^{p_1} \times L^{p_2} \rightarrow L^p$ bounds for this operator. Our results are optimal within the known boundedness range for the twisted paraproduct with a smooth kernel. 

\end{abstract}

\maketitle
\section{Introduction}

The \textit{triangular Hilbert transform} is a bilinear singular integral operator of the form
\[
H(f,g)(x,y)=\operatorname{p.v.} \int_{\R} f(x+t,y) g(x,y+t)\frac{dt}{t}, \quad (x,y)\in \R^2,
\]
initially defined for a pair of two-dimensional Schwartz functions $f$ and $g$. It is a major open problem whether the operator $H$ extends to a bounded operator from $L^{p_1}(\R^2) \times L^{p_2}(\R^2)$ into $L^p(\R^2)$ for some exponents $p_1$, $p_2$ and $p$. Such a result, if available, would unify and extend various prominent results in harmonic analysis, including Carleson's theorem \cite{Carleson1966,Fefferman1973,LaceyThiele2000} on the pointwise almost everywhere convergence of Fourier series of $L^p$ functions, as well as uniform bounds for the family of directional bilinear Hilbert transforms \cite{Thiele2002,GrafakosLi2004,Li2006,UraltsevWarchalski2022}. Partial progress towards the conjectured bounds for the triangular Hilbert transform was obtained by Kova\v{c}, Thiele and Zorin-Kranich~\cite{KTZ}, who studied the trilinear form associated with the dyadic triangular Hilbert transform and proved $L^p$ bounds for it under the assumption that one of the three input functions takes a special form. Additionally, cancellation estimates for the truncated triangular Hilbert transform were established by Zorin-Kranich~\cite{Z} and Durcik, Kova\v{c} and Thiele~\cite{DKT}. 

In this article, we will focus on a family of directional operators related to the triangular Hilbert transform. For any $\theta=(\theta_1,\theta_2) \in \R^2 \setminus \{(0,0)\}$, the corresponding operator $H_\theta$ acts on a pair of two-dimensional Schwartz functions $f$ and $g$ as
\[
H_\theta(f,g)(x,y)=\operatorname{p.v.} \int_{\R} f(x+\theta_1 t,y) g(x,y+\theta_2 t)\frac{dt}{t}, \quad (x,y)\in \R^2.
\]
When $\theta_1=0$, the operator $H_\theta$ reduces to the product of the function $f$ with the Hilbert transform applied to the function $g$ in the second variable. Consequently, $H_\theta$ extends to a bounded operator from $L^{p_1}(\R^2) \times L^{p_2}(\R^2)$ into $L^p(\R^2)$ whenever $1<p_1,p_2<\infty$, $0<p<\infty$ and $1/p_1+1/p_2=1/p$. A similar discussion applies to the case when $\theta_2=0$. Conversely, in the generic case when $\theta_1\neq 0$ and $\theta_2 \neq 0$, the problem of the $L^{p_1}(\R^2) \times L^{p_2}(\R^2) \rightarrow L^p(\R^2)$ boundedness of $H_\theta$ is equivalent to the same question for the triangular Hilbert transform $H=H_{(1,1)}$ and, as such, it remains open. 

In the absence of any $L^p$ bounds for the operators $H_\theta$, it is of interest to investigate whether we can bound at least a suitable superposition of such operators. Specifically, we will study averaged operators of the form
\begin{equation}\label{E:average}
A_\Omega(f,g)=\int_{\S^1} \Omega(\theta) H_\theta(f,g)\,d\sigma(\theta),
\end{equation}
where $\S^1$ denotes the unit sphere in $\R^2$ and $\Omega$ is a suitable function on $\S^1$. Since $H_{-\theta}=-H_\theta$ for any $\theta \in \S^1$, there is no loss of generality in assuming that $\Omega$ is odd. Setting formally $\Omega=\frac{1}{2}\delta_{\theta}-\frac{1}{2} \delta_{-\theta}$, where $\delta_z$ denotes the Dirac delta measure at the point $z$, we recover the original directional operator $H_\theta$. In this article, we establish $L^p$ estimates for the operator $A_\Omega$ under the assumption that $\Omega$ is an odd function belonging to $L^q(\S^1)$ for some $q>1$. Before stating our main result in detail, we discuss relevant known results to put our work into context.

For a suitable two-dimensional singular integral kernel $K$ and a pair of two-dimensional Schwartz functions $f$ and $g$, let us consider the operator 
\begin{equation}\label{E:twisted-paraproduct}
T_K(f,g)(x,y)=\operatorname{p.v.} \int_{\R^2} f(s,y)g(x,t)K(x-s,y-t) dsdt, \quad (x,y)\in \R^2.
\end{equation}
The operator $T_K$ is usually referred to as the \textit{twisted paraproduct}. Given a positive integer $L$, we say that $K$ is a \textit{Calder\'on-Zygmund kernel of order $L$} if it is a tempered distribution whose Fourier transform satisfies the smoothness condition 
\begin{equation}\label{E:sigma}
|\partial^{\alpha} \wh{K}(\xi,\eta)| \leq |(\xi,\eta)|^{-|\alpha|}, \quad (\xi,\eta)\in \R^2 \setminus \{(0,0)\},
\end{equation}
for all multiindices $\alpha$ such that $|\alpha|\leq L$. It follows from the results of Kova\v{c}~\cite{Kovac} that whenever $1<p_1,p_2<\infty$, $0<p<2$ satisfy $1/p_1+1/p_2=1/p$, then there are positive constants $C$ and $L$ such that for all Calder\'on-Zygmund kernels $K$ of order $L$, the following estimate holds:
\begin{equation}\label{E:kovac}
\|T_K(f,g)\|_{L^{p}(\mathbb{R}^{2})}\leq C \|f\|_{L^{p_{1}}(\mathbb{R}^{2})}\|g\|_{L^{p_{2}}(\mathbb{R}^{2})}\, .
\end{equation}
It remains an open problem whether inequality~\eqref{E:kovac} is valid also for some $p\geq 2$.

To connect the two operator families $A_\Omega$ and $T_K$, we observe that (a multiple of) $A_\Omega$ can be written in the form~\eqref{E:twisted-paraproduct}, where
\begin{equation}\label{E:kernel-omega}
K(x,y) = \dfrac{\Omega\left((x,y)/|(x,y)|\right)}{|(x,y)|^2}.
\end{equation}
The kernel $K$ from~\eqref{E:kernel-omega} is not a Calder\'on-Zygmund kernel unless we impose additional smoothness assumptions on the function $\Omega$. A prototypical example of an operator of the form as in~\eqref{E:twisted-paraproduct} and~\eqref{E:kernel-omega} which does not fall into the Calder\'on-Zygmund framework of~\cite{Kovac} is the operator
\begin{equation}\label{E:B}
B(f,g)(x,y)=\int_0^1 H_{(\alpha,1)}(f,g)(x,y) \,d\alpha \, ,
\end{equation}
studied by Durcik and Roos in~\cite{DR2021}. The operator $B$ is associated with the odd function
\[
\Omega(\theta_1,\theta_2)=\frac{\chi_{(0,\infty)}(\theta_2)-\chi_{(0,\infty)}(\theta_2-\theta_1)}{\theta_1^2}, \quad (\theta_1,\theta_2)\in \S^1,
\]
which is of bounded variation on $\S^1$. The corresponding kernel~\eqref{E:kernel-omega} is the kernel of Calder\'on's first commutator. 

In the present paper, we study boundedness properties of the twisted paraproduct~\eqref{E:twisted-paraproduct} associated with a kernel of the form~\eqref{E:kernel-omega}, where $\Omega$ is a function with vanishing integral on $\S^1$ which belongs to the space $L^q(\S^1)$ for some $q>1$. With a slight abuse of notation, this operator will be denoted by $T_\Omega$. Boundedness properties of the operator $A_\Omega$ from~\eqref{E:average} follow by specializing our results to the case when $\Omega$ is odd. 

\begin{theorem}\label{T:main-banach}
    Let $1<p_{1},p_{2}<\infty$, $1\leq p<2$ satisfy 
    $\frac{1}{p}=\frac{1}{p_{1}}+\frac{1}{p_{2}}$, and let $q>1$. Suppose that $\Omega \in L^{q}(\mathbb{S}^{1})$ fulfills $\int_{\S^{1}}\Omega(\theta)d\sigma (\theta)=0$. Then the operator $T_\Omega$ admits a bounded extension from $L^{p_1}(\R^2) \times L^{p_2}(\R^2)$ into $L^p(\R^2)$.
\end{theorem}

It is known that the triangular Hilbert transform is never bounded from $L^{p_1}(\R^2) \times L^{p_2}(\R^2)$ into $L^1(\R^2)$, see \cite[Appendix A.3]{KTZ} and \cite[Theorem 1.4]{L2025}. Nevertheless, Theorem~\ref{T:main-banach} shows that the averaged operator $A_\Omega$ remains bounded as long as $\Omega \in L^q(\S^1)$ for some $q>1$. In our second main result, we show that the operator $A_\Omega$ can be bounded from $L^{p_1}(\R^2) \times L^{p_2}(\R^2)$ into $L^p(\R^2)$ even for $p<1$, but more restrictive assumptions on the function $\Omega$ are required in this case. 

\begin{theorem}\label{T:main}
    Let $1<p_{1},p_{2}<\infty$, $0<p<1$ satisfy 
    $\frac{1}{p}=\frac{1}{p_{1}}+\frac{1}{p_{2}}$, and let $q>1$. Then the following conditions are equivalent.

    \textup{(i)} $\frac{1}{p}+\frac{1}{q}<2$;

    \textup{(ii)} the operator $T_\Omega$ admits a bounded extension from $L^{p_1}(\R^2) \times L^{p_2}(\R^2)$ into $L^p(\R^2)$ for every function $\Omega \in L^{q}(\mathbb{S}^{1})$ with $\int_{\S^{1}}\Omega(\theta)d\sigma (\theta)=0$.
\end{theorem}

The proofs of Theorems~\ref{T:main-banach} and~\ref{T:main} build on the techniques developed in~\cite{GHH,GHS,HP,DS2024} in connection with the study of the more standard Coifman-Meyer type bilinear singular integral operators with rough kernels. These techniques are then combined with known results on the boundedness properties of the twisted paraproduct with a Calder\'on-Zygmund kernel~\cite{Kovac} and with examples showing the unboundedness of the triangular Hilbert transform for $p=1$~\cite{KTZ,L2025}.
In addition, inspired by the recent preprint~\cite{BSpreprint} by Bhojak and Shrivastava, we investigate the methodological connections between our problem and the problem of the boundedness of the triangular Hilbert transform along curves~\cite{CDR2021,HLpreprint}.

To be more specific, the first main ingredient of our proofs, which is used to handle the ``high-frequency'' part of the operator $T_\Omega$, is a bilinear Fourier multiplier theorem developed in \cite{HLpreprint}. Originally, this multiplier theorem was introduced to prove smoothing inequalities for certain bilinear operators related to the triangular Hilbert transform along non-flat curves. Alternatively, bounds for the high-frequency part of $T_\Omega$ can be obtained via the wavelet decomposition technique of Grafakos, He and Honz\'ik~\cite{GHH}. We do not pursue the latter approach here but we briefly comment on it in Section~~\ref{S:smoothing} below.

The broader study of multilinear smoothing inequalities has recently gained considerable attention. For instance, a smoothing inequality for a bilinear averaging operator acting on one-dimensional functions and associated with certain curves was introduced by Christ~\cite{CZ2024}. Subsequently, Christ, Durcik and Roos~\cite{CDR2021} obtained a smoothing inequality for a bilinear averaging operator acting on two-dimensional functions and related to the triangular Hilbert transform along a parabola. Their approach relied on local Fourier series expansions, which in turn reduced the problem to sublevel set estimates. This strategy was also employed by Bhojak and Shrivastava in \cite{BSpreprint}, where the authors established a new smoothing inequality in order to provide an alternative proof of the boundedness of one-dimensional rough bilinear singular integrals of Coifman-Meyer type. 
In contrast, the framework used to prove the smoothing inequality in \cite{HLpreprint} is inspired by the LGC method developed by Gaitan and Lie \cite{GL2025}. A comparative analysis of these two distinct methodologies and of their relationship to the wavelet decomposition method from~\cite{GHH} remains a topic of independent interest.

The second main ingredient of our proofs is the realization that bounds for the ``mid-frequency'' part of the operator $T_\Omega$ can be deduced from logarithmic estimates for the family of \textit{shifted twisted paraproducts}, and that these estimates are in fact a direct consequence of bounds for the twisted paraproduct with a Calder\'on-Zygmund kernel. To be more specific, the family of shifted twisted paraproducts consists of operators of the form~\eqref{E:twisted-paraproduct} associated with the two-parameter family of kernels
\begin{equation}\label{E:kuv}
K^{u,v}(s,t)=\sum_{k\in \Z} 2^{-2k} \phi(2^{-k}s-u) \rho(2^{-k}t-v),
\end{equation}
where $(u,v)\in \R^2$ and $\phi$, $\rho$ are Schwartz functions with compact Fourier support such that $\wh{\phi}$ is supported away from the origin.
In order to establish Theorem~\ref{T:main-banach} for a tuple of exponents $(p_1,p_2,p)$, we need the following logarithmic-in-$u$ and polynomial-in-$v$ bound for the shifted operators:
\begin{equation}\label{E:logarithmic-blowup}
\|T_{K^{u,v}}(f,g)\|_{L^p(\R^2)} \leq C \log(2+|u|) (1+|v|)^{L}\|f\|_{L^{p_1}(\R^2)} \|g\|_{L^{p_2}(\R^2)}.
\end{equation}
By making use of the standard shifted square function estimates~\cite{Muscalu}, we show that~\eqref{E:logarithmic-blowup} holds whenever the twisted paraproduct is $L^{p_1}(\R^2) \times L^{p_2}(\R^2) \rightarrow L^{p}(\R^2)$ bounded for all Calder\'on-Zygmund kernels of order $L$ with norm depending only on $p_1$, $p_2$ and $L$. Consequently, if the available estimates for the twisted paraproduct with a Calder\'on-Zygmund kernel could be extended to a larger range of exponents, we would immediately get an extension of Theorem~\ref{T:main-banach} to the same range of exponents. A similar observation was made in~\cite{DR2021} in connection with the operator~\eqref{E:B}. 

We expect that the estimate~\eqref{E:logarithmic-blowup} can be improved from polynomial-in-$v$ to logarithmic-in-$v$. Such a result, however, is not needed for our applications, and we thus do not pursue it here. We point out that estimates for certain shifted quadrilinear
forms related to the shifted twisted paraproduct are the stepping stones towards proving quantitative convergence results for ergodic averages with respect to two commuting transformations, see the work by Durcik, Kova\v{c},  \v{S}kreb  and Thiele on double ergodic averages~\cite{DKST} and by Durcik and \v{S}kreb~\cite{DS2026} on cubic ergodic averages. We also direct the reader to the preprint~\cite{DBpreprint} by Becker and Durcik on the shifted bilinear Hilbert transform and its connections to double recurrence ergodic averages. The shifted estimates most related to ours are those from ~\cite{DS2026}; we refer to Section 4 for further details.

\medskip
A detailed overview of the proofs of Theorems~\ref{T:main-banach} and~\ref{T:main} is provided in Section~\ref{S:overview}. In that section, we decompose the operator $T_\Omega$ into various pieces and we formulate a number of propositions to control these pieces. The proofs of the corresponding propositions can then be found in Sections~\ref{S:smoothing}--\ref{S:counterexample}.

\subsection*{Acknowledgements}
We thank Georgios Dosidis for his work on this project in its early stages, in particular for drafting preliminary versions of the proofs of Lemmas~\ref{L:shifted} and~\ref{pr:fiberCZ} and for investigating the wavelet decomposition approach towards Proposition~\ref{P:smoothing-diagonal}. We also thank Ankit Bhojak, Polona Durcik and Saurabh Shrivastava for helpful discussions. The second author was partially supported by the grant no.\ 26-21107S of the Czech Science Foundation.


\section{Overview of the proof}\label{S:overview}

Let us fix a $q$ satisfying $1 < q \leq \infty$ and a function $\Omega\in L^q(\S^{1})$ with mean value zero. Let $K$ be the kernel given by~\eqref{E:kernel-omega}.
We fix a smooth function $\a$ in $\R^+_0$ such that $\a(t) = 1$ for $t \in [0,1]$, $0 < \a(t) < 1$ for $t \in (1,2)$ and $\a(t) = 0$ for $t \geq 2$. For $(y,z) \in\R^2$ and $\ell\in\Z$ we introduce the function
$$\b^\ell(y,z) = \a\left(2^{-\ell}|(y,z)|\right) - \a\left(2^{-\ell+1}|(y,z)|\right).$$
Since $\sum_{\ell \in \mathbb Z} \beta^\ell(y,z)=1$ whenever $(y,z) \neq (0,0)$, we may decompose the kernel $K$ as 
\[
K=\sum_{\ell\in \Z} K^\ell,
\]
where $K^\ell = \b^\ell K$. By the homogeneity of $K$, we have $K^\ell(y,z)=2^{-2\ell}K^0(2^{-\ell}(y,z))$.

Let $0 < \d < 1/q'$, where $q'$ is the H\"older conjugate of $q$, i.e.\ $1/q+1/q'=1$. We recall that the symbol $\wh{K^0}$ satisfies the decay estimate
\begin{equation}\label{E:decay-K0}
|\widehat{K^0}(\xi,\h)|\leq C\|\Omega\|_{L^q(\S^1)}|(\xi,\h)|^{-\d},
\end{equation}
see, e.g., \cite[Lemma 8.20]{Duo}. We also point out that this estimate in fact holds for any function $\Omega \in L^q(\S^1)$, without the need to require the mean zero condition. 

We proceed by decomposing the symbol $\wh{K^0}$. Let $\psi$ be a smooth function on $\R$ whose Fourier transform is given by
\[
\wh{\psi}(\xi)=\alpha(|\xi|)-\alpha(2|\xi|), \quad \xi \in \R.
\]
Then $\wh{\psi}$ is supported in $[-2,-1/2] \cup [1/2,2]$ and satisfies $\sum_{k\in \Z} \wh{\psi}(2^{-k} \xi)=1$ whenever $\xi \neq 0$. We write
\begin{equation*}
\wh{K^0}(\xi,\eta)=\sum_{k_1\in \Z} \sum_{k_2\in \Z} \wh{K^0}(\xi,\eta) \wh{\psi}(2^{-k_1}\xi) \wh{\psi}(2^{-k_2}\eta)
=\sum_{j\in \Z} \wh{K_j^0}(\xi,\eta),
\end{equation*}
where we define
\begin{equation}\label{E:kj}
\wh{K_j^0}(\xi,\eta)
=\sum_{\substack{(k_1,k_2)\in \Z^2\\\max\{k_1,k_2\}=j}}  \wh{K^0}(\xi,\eta) \wh{\psi}(2^{-k_1}\xi) \wh{\psi}(2^{-k_2}\eta).
\end{equation}
We then set
\[
K_j(y,z)=\sum_{\ell\in \Z} 2^{-2\ell} K^0_j(2^{-\ell}(y,z))
\]
and denote by $T_j$ the twisted paraproduct~\eqref{E:twisted-paraproduct} associated with the kernel $K_j$.

In order to prove Theorems~\ref{T:main-banach} and~\ref{T:main}, we establish $L^{p_1}(\R^2) \times L^{p_2}(\R^2) \rightarrow L^p(\R^2)$ bounds for the operators $T_j$ with norms that are summable in $j$. We first focus on the ``low-frequency'' case $j\leq 0$. 

\begin{proposition}\label{P:j<0}
Let $1<p_1,p_2<\infty$, $0<p<2$ satisfy $\frac{1}{p}=\frac{1}{p_1}+\frac{1}{p_2}$. Let $1<q\leq \infty$, $0<\delta<1/q'$, and let $j\leq 0$ be an integer. Then there exists a constant $C>0$ depending on $p_1$, $p_2$, $q$ and $\delta$ such that
\[
\|T_j\|_{L^{p_1}(\R^2)\times L^{p_2}(\R^2)\to L^p(\R^2)}\leq C 2^{j(1-\delta)} \|\Omega\|_{L^q(\S^1)}.
\]
\end{proposition}

The result of Proposition~\ref{P:j<0} is a consequence of the $L^{p_1}(\R^2) \times L^{p_2}(\R^2) \rightarrow L^p(\R^2)$ boundedness of the twisted paraproduct with a Calder\'on-Zygmund kernel~\cite{Kovac} and of the fact that the symbol $\wh{K_j}$ satisfies the estimate
\begin{equation}\label{E:symbol-estimate}
|\partial^\a \wh{K_j} (\xi,\h)|\leq C 2^{j(1-\d)} \|\Omega\|_{L^q(\S^1)}|(\xi,\h)|^{-|\a|}
\end{equation}
for all multiindices $\a\in\Z^2_{\geq 0}$. For the proof of~\eqref{E:symbol-estimate}, see~\cite[Proposition 3]{GHH}.

In what follows, we may thus assume that $j\geq 1$. 
We split $K_j^0$ further into its diagonal (or ``high-frequency'') part $K^{0,1}_j$ and off-diagonal (or ``mid-frequency'') part $K_j^{0,2}$ defined as
\begin{equation}\label{E:diagonal-symbol}
\wh{K_j^{0,1}}(\xi,\eta)=\sum_{k=1}^j \wh{K^0}(\xi,\eta) \wh{\psi}(2^{-j}\xi) \wh{\psi}(2^{-k}\eta)
+\sum_{k=1}^{j-1} \wh{K^0}(\xi,\eta) \wh{\psi}(2^{-k}\xi) \wh{\psi}(2^{-j}\eta)
\end{equation}
and
\[
\wh{K_j^{0,2}}(\xi,\eta)=\sum_{k=-\infty}^{0} \wh{K^0}(\xi,\eta) [\wh{\psi}(2^{-j}\xi) \wh{\psi}(2^{-k}\eta)
+ \wh{\psi}(2^{-k}\xi) \wh{\psi}(2^{-j}\eta)].
\]
We note that the second sum in~\eqref{E:diagonal-symbol} should be understood as zero if $j=1$. 
By setting
\begin{equation}\label{E:varphi}
\wh{\varphi}(\xi)=\sum_{k=-\infty}^{0} \wh{\psi}(2^{-k}\xi),
\end{equation}
we may write
\begin{equation}\label{E:off-diagonal-kernel}
\wh{K_j^{0,2}}(\xi,\eta)
=\wh{K^0}(\xi,\eta) [\wh{\psi}(2^{-j}\xi) \wh{\varphi}(\eta)
+ \wh{\varphi}(\xi) \wh{\psi}(2^{-j}\eta)].
\end{equation}
We then set
\[
K_j^1(y,z)=\sum_{\ell\in \Z} 2^{-2\ell} K^{0,1}_{j}(2^{-\ell}(y,z)), \quad K_j^2(y,z)=\sum_{\ell\in \Z} 2^{-2\ell} K^{0,2}_{j}(2^{-\ell}(y,z))
\]
and
denote by $T_j^1$ and $T_j^2$ the twisted paraproducts~\eqref{E:twisted-paraproduct} associated with the diagonal kernel $K_j^1$ and the off-diagonal kernel $K_j^2$, respectively. 

We start by establishing the $L^2(\R^2) \times L^2(\R^2) \rightarrow L^1(\R^2)$ boundedness of the operators $T_j^1$ with norms decaying in $j$. 

\begin{proposition}\label{P:smoothing-diagonal} 
Let $1<q\leq \infty$ and let $j\geq 1$ be an integer.
Then there exist constants $C>0$ and $\delta>0$ depending only on $q$ such that
\begin{equation}\label{E:diagonalL2}
	\|T_j^1\|_{L^{2}(\R^2)\times L^{2}(\R^2)\to L^1(\R^2)}\leq C 2^{-\delta j} \|\Omega\|_{L^q(\S^1)}.
        \end{equation}
\end{proposition}

Proposition~\ref{P:smoothing-diagonal} is proved in Section~\ref{S:smoothing} based on a bilinear Fourier multiplier theorem from~\cite{HLpreprint}. 

We will next establish a larger range of $L^{p_1}(\R^2) \times L^{p_2}(\R^2) \rightarrow L^p(\R^2)$ bounds for the operator $T_j^1$. 

\begin{proposition}\label{pr:shifted} Let $1<p_1,p_2<\infty$, $1< p <\infty$ satisfy $\frac1{p_1}+\frac{1}{p_2} = \frac1p$. Let $1<q\leq \infty$ and let $j\geq 1$ be an integer.
Then there exist constants $C>0$ and $\delta>0$ depending only on $p_1$, $p_2$ and $q$ such that 
	\begin{equation}
	\|T^1_j\|_{L^{p_1}(\R^2)\times L^{p_2}(\R^2)\to L^p(\R^2)}\leq C 2^{-\delta j} \|\Omega\|_{L^q(\S^1)}.
        \end{equation}
	\end{proposition} 

The proof of Proposition~\ref{pr:shifted} can be found in Subsection~\ref{sub:1}. It makes use of standard techniques involving shifted square functions, see, e.g., \cite{Muscalu, DR2021,DS2024,BSpreprint}, and of an interpolation argument utilizing Proposition~\ref{P:smoothing-diagonal}. Notably, bounds for the operator $T_j^1$ do not rely on repeated applications of the Cauchy-Schwarz inequality and telescoping identities, which is the usual method for proving boundedness of the twisted paraproduct with a smooth kernel. For this reason, the boundedness range in Proposition~\ref{pr:shifted} is larger than the one in which bounds for the twisted paraproduct with a Calder\'on-Zygmund kernel are currently known. 

We will now focus on proving bounds for the operator $T_j^2$. 

\begin{proposition}\label{P:shifted-decay} 
Let $1<p_1,p_2<\infty$, $1\leq p < 2$ satisfy $\frac1{p_1}+\frac{1}{p_2} = \frac1p$. Let $1<q\leq \infty$ and let $j\geq 1$ be an integer. Then there exist constants $C>0$ and $\delta>0$ depending only on $p_1$, $p_2$ and $q$ such that
	\begin{equation}\label{E:off-diagonalLp}
	\|T^2_j\|_{L^{p_1}(\R^2)\times L^{p_2}(\R^2)\to L^p(\R^2)}\leq C 2^{-\delta j} \|\Omega\|_{L^q(\S^1)}.
        \end{equation}
\end{proposition}

Proposition~\ref{P:shifted-decay} is proved in Subsection~\ref{sub:2}. Our argument relies on logarithmic estimates for the family of shifted twisted paraproducts, which are established in the same subsection as well. 


A combination of Propositions~\ref{pr:shifted} -- \ref{P:shifted-decay} yields the following corollary.

\begin{corollary}\label{P:interpolation-banach}
Let $1<p_1,p_2<\infty$, $1<p < 2$ satisfy $\frac1{p_1}+\frac{1}{p_2} = \frac1p$.
Let $1<q\leq \infty$ and let $j\geq 1$ be an integer.
Then there exist constants $C>0$ and $\delta>0$ depending only on $p_1$, $p_2$ and $q$ such that
\begin{equation}
	\|T_j\|_{L^{p_1}(\R^2)\times L^{p_2}(\R^2)\to L^p(\R^2)}\leq C 2^{-\delta j} \|\Omega\|_{L^q(\S^1)}.
        \end{equation}
\end{corollary}


We next focus on the $L^{p_1}(\R^2) \times L^{p_2}(\R^2) \rightarrow L^p(\R^2)$ boundedness of the operator $T_j$ in the case when $p\leq 1$. 

\begin{proposition}\label{P:interpolation-quasi-banach}
Let $1<p_1,p_2<\infty$, $0< p \leq 1$ satisfy $\frac1{p_1}+\frac{1}{p_2} = \frac1p$.
Let $1<q\leq \infty$ and let $j\geq 1$ be an integer. Assume that $\frac{1}{p}+\frac{1}{q}<2$.
Then there exist constants $C>0$ and $\delta>0$ depending only on $p_1$, $p_2$ and $q$ such that
\begin{equation}\label{E:decay-quasibanach}
	\|T_j\|_{L^{p_1}(\R^2)\times L^{p_2}(\R^2)\to L^p(\R^2)}\leq C 2^{-\delta j} \|\Omega\|_{L^q(\S^1)}.
        \end{equation}
\end{proposition}

Proposition~\ref{P:interpolation-quasi-banach} is proved in Section~\ref{S:CZ} via a fiberwise Calder\'on-Zygmund decomposition, combining the techniques of Bernicot~\cite{Bernicot} with those of He and Park~\cite{HP} and Bhojak and Shrivastava~\cite{BSpreprint}. 

Theorem~\ref{T:main-banach} is now a direct consequence of Proposition~\ref{P:j<0}, Corollary~\ref{P:interpolation-banach} and the case $p=1$ of Proposition~\ref{P:interpolation-quasi-banach} (as the assumption $1/p+1/q<2$ is automatically satisfied in this case). 
The implication $\textup{(i)} \Rightarrow \textup{(ii)}$ of Theorem~\ref{T:main} follows by combining Propositions~\ref{P:j<0} and~\ref{P:interpolation-quasi-banach}. The reverse implication is then a consequence of the following proposition. 

\begin{proposition}\label{P:counterexample}
Let $1<p_1,p_2<\infty$, $0< p < \infty$ satisfy $\frac1{p_1}+\frac{1}{p_2} = \frac1p$.
Let $1<q\leq \infty$ and assume that $\frac{1}{p}+\frac{1}{q}\geq 2$. Then there is an odd function $\Omega \in L^q(\S^1)$ for which the operator $T_\Omega$ is unbounded from $L^{p_1}(\R^2) \times L^{p_2}(\R^2)$ into $L^p(\R^2)$. 
\end{proposition}

Proposition~\ref{P:counterexample} is proved in Section~\ref{S:counterexample}. Our examples are inspired by the corresponding examples for the Coifman-Meyer type operators with rough kernels~\cite{GHHP,DS2024} and by the examples showing the failure of the $L^{p_1}(\R^2) \times L^{p_2}(\R^2) \rightarrow L^1(\R^2)$ boundedness of the triangular Hilbert transform~\cite{KTZ,L2025}.

\medskip
\textbf{Notation.} We use the symbol $A \lesssim B$ to indicate that $A \leq CB$ for some constant
$C>0$ that may vary from line to line. The expression $A \gtrsim B$ is defined analogously. We write $A \approx B$ if $A \lesssim B$ and $A \gtrsim B$ hold simultaneously.
For a function $f$ on $\R^d$, $d\geq 1$, and for $\lambda>0$, we denote $f_\lambda=\lambda^d f(\lambda \cdot)$. If $f$ is a function on $\R^2$ and $\phi$ is a function on $\R$, their partial convolutions are given by
\begin{equation}\label{E:convolution-first}
f \ast^{(1)} \phi(x,y)=\phi \ast^{(1)} f(x,y)=\int_{\R} f(x-t,y) \phi(t)\,dt,
\end{equation}
\begin{equation}\label{E:convolution-second}
f \ast^{(2)} \phi(x,y)=\phi \ast^{(2)} f(x,y)=\int_{\R} f(x,y-t) \phi(t)\,dt,
\end{equation}
assuming that the right-hand sides of~\eqref{E:convolution-first} and~\eqref{E:convolution-second} are well defined.

\section{Proof of Proposition~\ref{P:smoothing-diagonal}: a smoothing inequality}\label{S:smoothing}

In order to prove Proposition~\ref{P:smoothing-diagonal}, we will first consider a single-scale variant of the operator $T_j^1$. This operator will be denoted by $T_{j}^{0,1}$ and defined as the twisted paraproduct~\eqref{E:twisted-paraproduct} associated with the kernel $K_j^{0,1}$. We will establish the following smoothing inequality for the operator $T_{j}^{0,1}$.

\begin{lemma}\label{P:smoothing}
Let $1<q\leq \infty$.
There exists a constant $\delta>0$ depending only on $q$ such that
	\begin{equation}
	\|T_j^{0,1}\|_{L^{2}(\R^2)\times L^{2}(\R^2)\to L^1(\R^2)}\lesssim 2^{-\delta j} \|\Omega\|_{L^q(\S^1)}.
        \end{equation}
\end{lemma}

We denote by $\mathcal{F}^{(1)}f(\xi,y)=\int_{\mathbb{R}} f(x,y)e^{-2\pi ix\xi}dx$
  the Fourier transform in the first variable for $f\in L^1(\R^2)$  and by $\mathcal{F}^{(2)}f(x,\eta)=\int_{\mathbb{R}} f(x,y)e^{-2\pi iy\eta}dy$ the one in the second variable. Let $T_K$ be the twisted paraproduct defined in~\eqref{E:twisted-paraproduct}.
We remark that for $m=\widehat K$, the operator $T_K$ can be written as
\begin{equation}\label{E:twisted-symbol}
T_K(f,g)(x,y)=\int_{\R^2} \mathcal{F}^{(1)}f(\xi,y) \mathcal{F}^{(2)}g(x,\h) m(\xi,\h) e^{2\pi i (x\xi+y\h)}d\xi d\h.
\end{equation}
Throughout this section, we occasionally denote the operator~\eqref{E:twisted-symbol} as $\widetilde{T}_m$ and call it ``the twisted paraproduct associated with the symbol $m$''. We point out that if our interest lies in establishing $L^2(\R^2) \times L^2(\R^2) \rightarrow L^1(\R^2)$ bounds for the operator $\widetilde{T}_m$ (which is the case in this section), then Plancherel's identity allows us to drop the partial Fourier transforms in~\eqref{E:twisted-symbol}.

Lemma~\ref{P:smoothing} follows as a consequence of the following bilinear Fourier multiplier theorem from \cite{HLpreprint}.

\begin{theorem}[Hsu, Lin {\cite[Theorem 1.6]{HLpreprint}}]\label{thm_universal_multiplier_bound}
    For a symbol \(m\in L^\infty(\R^2)\) and functions \(F,G\in (L^1\cap L^\infty)(\R^2)\), consider the fiberwise bilinear Fourier integral operator $\mathcal T_m$ 
    \begin{equation}\label{carlyT}
        \mathcal T_m(F,G)(x,y):=
        1_{[0,1]^2}(x,y)
        \int_{\R^2}
            F(\xi,y)
            G(x,\eta)
            m(\xi,\eta)
            e^{2\pi i(x\xi+y\eta)}
        d\xi d\eta.
    \end{equation}
    The following estimate holds:
    \begin{equation*}
        \|\mathcal T_m(F,G)\|_{L^1(\R^2)}
        \lesssim
        \|m\|^{\frac{1}{5}}_u
        \|m\|^{\frac{4}{5}}_U
        \|F\|_{L^2(\R^2)}
        \|G\|_{L^2(\R^2)},
    \end{equation*}
    where the two quantities $\|m\|_{u}$ and $\|m\|_{U}$ are defined by the following formulas:
    \begin{equation}\label{mlittleunorm}
\|m\|_{u}:=\left\|\mathcal{F}^{(1)}(\mathcal{D}_{(0,s)}m)(x,\eta)  \right\|^{\frac{1}{2}}_{L^{\infty}_{s\eta}(\mathbb{R}^{2})L^{2}_{x}([-1,1])} ,
    \end{equation}
    \begin{equation}\label{mcapitalunorm}
\|m\|_{U}:=\left\|\int_{\mathbb{R}}\mathcal{D}_{(0,s)}\mathcal{D}_{(u,v)}m(\xi, \eta)d\xi   \right\|^{\frac{1}{4}}_{L^{\infty}_{u\eta}(\mathbb{R}^{2})L^{1}_{s}(\mathbb{R})L^{2}_{v}(\mathbb{R})},
    \end{equation}
    where the multiplicative derivative $\mathcal{D}_{(y_{1},y_{2})}$ is defined by
    \begin{equation}\label{multideri_def}
(\mathcal{D}_{(y_{1},y_{2})}f)(x_{1},x_{2}):=f(x_{1}+y_{1} ,x_{2}+y_{2})\overline{f(x_{1},x_{2})}\, ,
    \end{equation}
    and the iterated Lebesgue norm is defined by
\begin{equation}
    \|f(x,y)\|_{L^{p}_{x}(A)L^{q}_{y}(B)}:=\left\| \|f(x,y)\|_{L^{p}_{x} (A)} \right\|_{L^{q}_{y}(B)}\, .
\end{equation}
\end{theorem}

We will apply Theorem~\ref{thm_universal_multiplier_bound} in the form of the following corollary.

\begin{corollary}\label{C:corollary-multiplier}
Let $\lambda>0$. Assume that $m\in L^\infty(\R^2)$ is supported in the set of those $(\xi,\eta) \in \R^2$ for which $|(\xi,\eta)| \leq \lambda$. Then the operator $\mathcal T_m$ defined by~\eqref{carlyT} satisfies
\[
\|\mathcal T_m\|_{L^2(\R^2) \times L^2(\R^2) \rightarrow L^1(\R^2)} \lesssim \lambda^{\frac{3}{5}} \|m\|_{L^\infty(\R^2)}.
\]
\end{corollary}

\begin{proof}
We observe that
\[
\mathcal{D}_{(0,s)}m(\xi,\eta)=m(\xi,\eta+s) \overline{m(\xi,\eta)}
\]
and
\[
\mathcal{D}_{(0,s)} \mathcal{D}_{(u,v)} m(\xi,\eta)
= m(\xi+u,\eta+v+s) \overline{m(\xi,\eta+s)} \overline{m(\xi+u,\eta+v)} m(\xi,\eta).
\]
Consequently,
\[
\|m\|_{u} \lesssim \left\|\int_{|(\xi,\eta)|\leq \lambda}\|m\|_{L^\infty(\R^2)}^{2}d\xi\right\|^{1/2}_{L^{\infty}_{s\eta}(\mathbb{R}^{2})L^{2}_{x}([-1,1])}\lesssim \lambda^{\frac{1}{2}} \|m\|_{L^\infty(\R^2)}
\]
and
\[
\|m\|_{U} \lesssim \left\|\int_{|(\xi,\eta)|\leq \lambda} \|m\|_{L^\infty(\R^2)}^4 d\xi\right\|^{\frac{1}{4}}_{L^{\infty}_{u\eta}(\mathbb{R}^{2})L^{1}_{s}(-2\lambda,2\lambda)L^{2}_{v}(-2\lambda,2\lambda)}\lesssim \lambda^{\frac{5}{8}}\|m\|_{L^\infty(\R^2)}.
\]
 Note that since $m$ is supported in $|(\xi,\eta)|\leq \lambda$, $L^{1}_{s}$, $L^{2}_{v}$ are integrated over $[-2\lambda,2\lambda]$.
The conclusion then follows from Theorem~\ref{thm_universal_multiplier_bound}.
\end{proof}

Let $k\in \N$.
In what follows, we denote by $C^k(\R^2)$ the space of $k$-times continuously differentiable functions on $\R^2$, equipped with the norm
\[
\|f\|_{C^k(\R^2)} = \max_{|\alpha|\leq k} \|\partial^\alpha f\|_{L^\infty(\R^2)},
\]
where the maximum is taken over all multiindices $\alpha$ of size at most $k$.

\begin{proof}[Proof of Lemma~\ref{P:smoothing}]
By symmetry, it suffices to establish the lemma for the twisted paraproduct associated with the symbol
\[
\sum_{k=1}^j \wh{K^0}(\xi,\eta) \wh{\psi}(2^{-j}\xi) \wh{\psi}(2^{-k}\eta).
\]
In addition, it is in fact enough to prove the desired bound for each of the pieces $\wh{K^0}(\xi,\eta) \wh{\psi}(2^{-j}\xi) \wh{\psi}(2^{-k}\eta)$ separately, as the exponential decay in $j$ will compensate for the blowup in the number of pieces. In what follows, we will thus fix $k\in \{1,\dots,j\}$ and focus on the corresponding single piece. As the $L^2(\R^2) \times L^2(\R^2) \rightarrow L^1(\R^2)$ norm of the twisted paraproduct is invariant under anisotropic scaling of the symbol, 
it is enough to consider the rescaled symbol
\[
m(\xi,\eta)=\wh{K^0}(\xi,2^{-j+k}\eta) \wh{\psi}(2^{-j}\xi) \wh{\psi}(2^{-j}\eta).
\]
The twisted paraproduct associated with this symbol will be denoted by $T^{0}_{j,k}$.

The symbol $m$ is supported in the set of those $(\xi,\eta) \in \R^2$ for which $|(\xi,\eta)| \approx 2^j$. Using~\eqref{E:decay-K0}, we thus deduce that for $\varepsilon \in (0,1/q')$,
\begin{equation*}
    \|m\|_{L^{\infty}(\R^2)}\lesssim 2^{-j(\frac{1}{q'}-\varepsilon)}\|\Omega\|_{L^{q}(\S^1)}. 
\end{equation*}
We point out that this inequality does not require the mean zero condition on $\Omega$, and thus holds for any $\Omega \in L^q(\S^1)$. 
Corollary~\ref{C:corollary-multiplier} then implies 
\begin{equation}\label{E:decay-q-5/2}
    \|\mathcal T_{m}(f,g)\|_{L^{1}(\R^2)}\lesssim 2^{j(\frac{1}{q}-\frac{2}{5}+\varepsilon)}\|\Omega\|_{L^{q}(\S^1)}\|f\|_{L^{2}(\R^2)}\|g\|_{L^{2}(\R^2)}
\end{equation}
for any pair of Schwartz functions $f$ and $g$.
When $q>\frac{5}{2}$, we can choose $\varepsilon>0$ in such a way that the right-hand side of~\eqref{E:decay-q-5/2} decays exponentially in $j$.
To go below $5/2$, we fix $f$, $g$, view $\mathcal T_m(f,g)$ as a linear operator in $\Omega$ and use the Riesz-Thorin interpolation theorem to interpolate between~\eqref{E:decay-q-5/2} and the trivial bound
\begin{equation*}
     \|\mathcal T_{m}(f,g)\|_{L^{1}(\R^2)}\leq
     \|\widecheck{m}\|_{L^1(\R^2)} \|f\|_{L^{2}(\R^2)}\|g\|_{L^{2}(\R^2)}
     \lesssim \|\Omega\|_{L^{1}(\S^1)}\|f\|_{L^{2}(\R^2)}\|g\|_{L^{2}(\R^2)}.
\end{equation*}
Then we obtain that for $q\in (1,\infty ]$, there is a constant $\delta=\delta(q)>0$ such that
\begin{equation}\label{eq_smoothing_local}
     \|\mathcal T_{m}(f,g)\|_{L^{1}(\R^2)}\lesssim 2^{-j\delta(q)}\|\Omega\|_{L^{q}(\S^1)}\|f\|_{L^{2}(\R^2)}\|g\|_{L^{2}(\R^2)}\, .
\end{equation}

When comparing the operators $\mathcal T_{m}$ and $T^{0}_{j,k}$, we
 observe that $\mathcal T_{m}$ contains an additional compact support $1_{[0,1]^{2}}(x,y)$. We will now demonstrate how to deduce the general smoothing inequality of Lemma~\ref{P:smoothing} from this compactly supported version. This is achieved through a standard procedure involving a commutator estimate ~\cite{CDR2021,BSpreprint}. 

Define the partial Littlewood-Paley operators
\begin{equation*}
\Delta_{j}^{(1)}F(x,y):=\psi_{2^j}\ast^{(1)}F(x,y), 
\quad
\Delta_{j}^{(2)}F(x,y):=\psi_{2^j}\ast^{(2)}F(x,y).
\end{equation*}
In addition, let $\widetilde{T}^0_{j,k}$ be the twisted paraproduct associated with the symbol $\wh{K^0}(\xi,2^{-j+k} \eta)$. We observe that
\begin{equation}\label{E:littlewood-paley-partial}
T^0_{j,k}(f,g)=\widetilde{T}^0_{j,k}(\Delta_{j}^{(1)}f,\, \Delta_{j}^{(2)}g).
\end{equation}
 Let $\eta$ be a smooth non-negative bump function supported in a small neighborhood of $[-\frac{1}{2},\frac{1}{2}]^{2}$ such that $\sum_{n\in\mathbb{Z}^{2}}\eta_{n}=1$, where $\eta_{n}(z)=\eta (z-n)$. With the help of~\eqref{E:littlewood-paley-partial}, we estimate 
\begin{equation*}
    \|T^{0}_{j,k}(f,g)\|_{L^{1}(\R^2)}\leq \sum_{n\in \mathbb{Z}^{2}}\int_{\mathbb{R}^2}\left|\widetilde{T}^0_{j,k}(\Delta_{j}^{(1)}f,\, \Delta_{j}^{(2)}g) \right|\eta_{n}\, .
\end{equation*}
Expanding the term $\widetilde{T}^0_{j,k}(\Delta_{j}^{(1)}f,\, \Delta_{j}^{(2)}g)(x,y)\eta_{n}(x,y)$, we obtain
\begin{equation*}
    \int_{\mathbb{R}^{2}} (\Delta_{j}^{(1)}f)(x-s,y)(\Delta_{j}^{(2)}g)(x,y-t)\eta_{n}(x,y)2^{j-k}K^{0}(s,2^{j-k}t)dsdt \, .
\end{equation*}
Note that $K^{0}(s,2^{j-k}t)$ is supported in the set where $|(s,t)|\lesssim 1$. Hence, we may insert two bump functions $\widetilde{\eta}_n(x-s,y)$, $\widetilde{\eta}_{n}(x,y-t)$ both equal to $1$ on the support of $\eta_{n}(x,y)K^{0}(s,2^{j-k}t)$ with the property $\|\widetilde{\eta}_{n}\|_{C^{1}(\R^2)}\lesssim 1$ and $\sum_{n\in \mathbb{Z}^2}\widetilde{\eta}_{n}\lesssim 1$. Therefore,
\begin{equation*}
    \|T^{0}_{j,k}(f,g)\|_{L^{1}(\R^2)}\leq  \sum_{n\in \mathbb{Z}^{2}}\int_{\mathbb{R}^2}\left|\widetilde{T}^0_{j,k}(\widetilde{\eta}_{n}\Delta_{j}^{(1)}f,\, \widetilde{\eta}_{n}\Delta_{j}^{(2)}g) \right|\eta_{n}
\end{equation*}
\begin{equation*}
    \leq  \sum_{n\in \mathbb{Z}^{2}}\int_{\mathbb{R}^2}\left|\widetilde{T}^0_{j,k}(\Delta_{j}^{(1)}(\widetilde{\eta}_{n}f),\, \Delta_{j}^{(2)}(\widetilde{\eta}_{n}g)) \right|\eta_{n}
\end{equation*}
\begin{equation*}
    + \sum_{n\in \mathbb{Z}^{2}}\int_{\mathbb{R}^2}\left|\widetilde{T}^0_{j,k}\left([\widetilde{\eta}_{n},\Delta_{j}^{(1)}]f,\, \widetilde{\eta}_{n}\Delta_{j}^{(2)}g\right) \right|\eta_{n}
\end{equation*}
\begin{equation*}
    +\sum_{n\in \mathbb{Z}^{2}}\int_{\mathbb{R}^{2}}\left|\widetilde{T}^0_{j,k}\left(\Delta_{j}^{(1)}(\widetilde{\eta}_{n}f),\, [\widetilde{\eta}_{n},\Delta_{j}^{(2)}]g\right) \right|\eta_{n}
\end{equation*}
\begin{equation*}
    =I_{1}+I_{2}+I_{3}\, ,
\end{equation*}
where the commutator is defined for $l=1,2$ as follows:
\begin{equation*}
[\widetilde{\eta}_{n},\Delta_{j}^{(l)}]F:=\widetilde{\eta}_{n}\Delta_{j}^{(l)}F-\Delta_{j}^{(l)}(\widetilde{\eta}_{n}F)\, .
\end{equation*}
The functions $\widetilde{\eta}_n f$ and $\widetilde{\eta}_n g$ are supported in a square of side-length roughly one, so
for $I_{1}$ we may directly apply \eqref{eq_smoothing_local} and obtain
\begin{equation*}
    |I_{1}|\lesssim 2^{-\delta j}\sum_{n\in \mathbb{Z}^{2}}\|\widetilde{\eta}_{n}f\|_{L^{2}(\R^2)}\|\widetilde{\eta}_{n}g\|_{L^{2}(\R^2)}\lesssim 2^{-\delta j}\|f\|_{L^{2}(\R^2)}\|g\|_{L^{2}(\R^2)}.
\end{equation*}
To estimate $I_{2}$, we use the mean value theorem and the rapid decay of the Schwartz function $\psi$ to deduce
\begin{equation*}
     \left|[\widetilde{\eta}_{n},\Delta_{j}^{(1)}]f\right|\lesssim 2^{-j}\|\widetilde{\eta}_{n}'\|_{L^{\infty}(\R^2)} \phi_{2^j}\ast^{(1)}|f|\lesssim 2^{-j}\phi_{2^j}\ast^{(1)}|f|\,,
\end{equation*}
where $\phi(t)=(1+|t|)^{-10}$. 
Hence, we have
\begin{equation*}
    I_{2}\lesssim 2^{-j}\int_{\R^2} 2^{j-k}|K^{0}(s,2^{j-k}t)|\int_{\R^2}(\phi_{2^j}\ast^{(1)}|f|) (x-s,y)(\phi_{2^j}\ast^{(2)}|g|) (x,y-t)
\end{equation*}
\begin{equation*}
\times \sum_{n\in\mathbb{Z}^{2}}\eta_{n}(x,y)dxdy\, dsdt\, .
\end{equation*}
Then by $\sum_{n\in\mathbb{Z}^{2}}\eta_{n}(x,y)= 1$ and the Cauchy-Schwarz inequality in the inner integral ($x,y$ variables), we have
\begin{equation*}
    I_{2}\lesssim 2^{-j}\int_{\R^2} 2^{j-k}|K^{0}(s,2^{j-k}t)|\cdot \|\phi_{2^j}{\ast}^{(1)}|f|\|_{L^{2}(\R^2)}\cdot \|\phi_{2^j}{\ast}^{(2)}|g|\|_{L^{2}(\R^2)}dsdt
\end{equation*}
\begin{equation*}
    \lesssim 2^{-j}\|K^{0}\|_{L^{1}(\R^2)}\|f\|_{L^{2}(\R^2)}\|g\|_{L^{2}(\R^2)}\lesssim 2^{-j}\|\Omega\|_{L^{1}(\S^1)}\|f\|_{L^{2}(\R^2)}\|g\|_{L^{2}(\R^2)}\, .
\end{equation*}
The term $I_{3}$ can be estimated similarly. Combining the estimates for $I_{1},I_{2},I_{3}$, we complete the proof of Lemma \ref{P:smoothing}.
\end{proof}

\begin{remark}
Using the wavelet decomposition technique introduced in~\cite{GHH} and further refined in~\cite{GHS}, one may prove the following analogue of~\cite[Theorem 1.3]{GHS}.

\textit{Let $1\leq q <4$ and let $m$ be a function in $L^q(\R^2) \cap C^{M_q}(\R^2)$, where $M_q=\lfloor \frac{2}{4-q}\rfloor+1$.
Then 
\begin{equation}\label{E:wavelet}
\|\widetilde{T}_m\|_{L^2(\R^2) \times L^2(\R^2) \rightarrow L^1(\R^2)} \lesssim \|m\|_{C^{M_q}(\R^2)}^{1-\frac{q}{4}} \|m\|_{L^q(\R^2)}^{\frac{q}{4}}.
\end{equation}}

The above result is obtained by mimicking the proof of~\cite[Theorem 1.3]{GHS}, applying the estimates fiberwise where appropriate. In particular, if $m$ is a function supported in the set of those $(\xi,\eta) \in \R^2$ for which $|(\xi,\eta)| \leq \lambda$, then inequality~\eqref{E:wavelet} applied with $q=1$ yields
\begin{equation}\label{E:second-multiplier-theorem}
\|\widetilde{T}_m\|_{L^2(\R^2) \times L^2(\R^2) \rightarrow L^1(\R^2)} \lesssim \lambda^{\frac{1}{2}} \|m\|_{C^1(\R^2)}.
\end{equation}
This estimate can then be used in place of Corollary~\ref{C:corollary-multiplier} to provide an alternative proof of Lemma~\ref{P:smoothing}. The presence of the $L^\infty$-norm of derivatives of the function $m$ on the right-hand side of~\eqref{E:second-multiplier-theorem} does not cause any issues here as a variant of the decay estimate~\eqref{E:decay-K0} holds for the derivatives of $\wh{K^0}$. 
\end{remark}

\begin{proof}[Proof of Proposition~\ref{P:smoothing-diagonal}]
Let $k_1$, $k_2$ be positive integers such that $\max\{k_1,k_2\}=j$. As there are only $2j-1$ such pairs of integers, it suffices to prove inequality~\eqref{E:diagonalL2} with $T_j^1$ replaced by the twisted paraproduct associated with the symbol
\[
\sum_{\ell \in \Z} \wh{K^0}(2^\ell(\xi,\eta)) \wh{\psi}(2^{\ell-k_1}\xi) \wh{\psi}(2^{\ell-k_2}\eta).
\]
Throughout the rest of the proof, we denote this operator simply by $T$.

We introduce the symbol
\[
m(\xi,\eta)=\wh{K^0}(\xi,\eta) \wh{\psi}(2^{-k_1}\xi) \wh{\psi}(2^{-k_2}\eta).
\]
We observe that if $\rho$ is a smooth function on $\R$ whose Fourier transform is supported in $[-4,-1/4] \cup [1/4,4]$ and is equal to $1$ on $[-2,-1/2] \cup [1/2,2]$ then 
\[
m(\xi,\eta)=m(\xi,\eta) \wh{\rho}(2^{-k_1}\xi) \wh{\rho}(2^{-k_2}\eta).
\]
Our operator can thus be rewritten as
\[
T(f,g)
=\sum_{\ell\in \Z} \widetilde{T}_{m(2^\ell \cdot)} (f,g)
=\sum_{\ell\in \Z} \widetilde{T}_{m(2^\ell \cdot)}(f \ast^{(1)} \rho_{2^{k_1-\ell}},g \ast^{(2)} \rho_{2^{k_2-\ell}}).
\]
We recall that the $L^2(\R^2) \times L^2(\R^2) \rightarrow L^1(\R^2)$ norm of the twisted paraproduct is invariant under rescaling of the symbol, and therefore
\[
\|\widetilde{T}_{m(2^l \cdot)}\|_{L^2(\R^2) \times L^2(\R^2) \rightarrow L^1(\R^2)}
=\|\widetilde{T}_{m}\|_{L^2(\R^2) \times L^2(\R^2) \rightarrow L^1(\R^2)}.
\]
By Lemma~\ref{P:smoothing}, there is $\delta>0$ such that
\[
\|\widetilde{T}_{m}\|_{L^2(\R^2) \times L^2(\R^2) \rightarrow L^1(\R^2)} \lesssim 2^{-j\delta} \|\Omega\|_{L^q(\S^1)}. 
\]
Thus,
\begin{align*}
\|T(f,g)\|_{L^1(\R^2)}
&\leq \sum_{\ell\in \Z} \|\widetilde{T}_{m(2^\ell\cdot)}(f \ast^{(1)} \rho_{2^{k_1-\ell}},g \ast^{(2)} \rho_{2^{k_2-\ell}})\|_{L^1(\R^2)}\\
&\lesssim 2^{-j\delta} \|\Omega\|_{L^q(\S^1)} \sum_{\ell\in \Z}\|f \ast^{(1)} \rho_{2^{k_1-\ell}}\|_{L^2(\R^2)} \|g \ast^{(2)} \rho_{2^{k_2-\ell}}\|_{L^2(\R^2)}\\
&\lesssim 2^{-j\delta} \|\Omega\|_{L^q(\S^1)} \left(\sum_{\ell\in \Z}\|f \ast^{(1)} \rho_{2^{k_1-\ell}}\|_{L^2(\R^2)}^2 \right)^{\frac{1}{2}} \left(\sum_{\ell\in \Z} \|g \ast^{(2)} \rho_{2^{k_2-\ell}}\|_{L^2(\R^2)}^2 \right)^{\frac{1}{2}}\\
&\lesssim 2^{-j\delta} \|\Omega\|_{L^q(\S^1)} \|f\|_{L^2(\R^2)} \|g\|_{L^2(\R^2)}.
\end{align*}
We note that the third inequality in the above chain follows by the Cauchy-Schwarz inequality, and the fourth one by Plancherel's identity and by the fact that the Fourier supports of the functions $f \ast^{(1)} \rho_{2^{k_1-\ell}}$ (or $g \ast^{(2)} \rho_{2^{k_2-\ell}}$, respectively) have bounded overlap in $\ell\in \Z$. 
\end{proof}

\section{Proofs of Propositions~\ref{pr:shifted} and~\ref{P:shifted-decay}: estimates for shifted operators}\label{S:shifted}

Let $\phi$ be a Schwartz function on $\R$ whose Fourier transform is compactly supported away from the origin. Throughout this section, for $u>0$ and $v\in \R$ we will use the notation
\[
\phi_u^v(x)=u\phi(ux-v), \quad x\in \R.
\]
It is well known (see, e.g., \cite[Theorem 5.1]{Muscalu}) that if $1<r<\infty$, then there is a constant $C>0$ depending on $r$ and $\phi$ such that
\begin{equation}\label{E:shifted-square-function}
\left\|\left(\sum_{k\in \Z} |f \ast \phi_{2^{k}}^v|^2\right)^{\frac{1}{2}}\right\|_{L^r(\R)} \leq C \log(2+|v|) \|f\|_{L^r(\R)}
\end{equation}
for any $v\in \R$ and any Schwartz function $f$ on $\R$. The estimate~\eqref{E:shifted-square-function} will play an important role in the proofs of Propositions~\ref{pr:shifted} and~\ref{P:shifted-decay} in Subsections~\ref{sub:1} and~\ref{sub:2} below.

\subsection{Proof of Proposition~\ref{pr:shifted}}\label{sub:1}

We start by establishing a variant of Proposition~\ref{pr:shifted} involving a power-type blow-up in $j$. The proof of such an estimate is standard and resembles the arguments from~\cite{Muscalu}, \cite[Section 2.1]{DR2021}, \cite[proof of Proposition 4]{DS2024}, \cite[proof of Lemma 5.3]{BSpreprint}; we include it below for the sake of completeness.

\begin{lemma}\label{L:shifted} Let $1<p_1,p_2<\infty$, $1\leq p <\infty$ satisfy $\frac1{p_1}+\frac{1}{p_2} = \frac1p$, and let $j\geq 1$ be an integer. Then 
	\begin{equation}\label{E:diagonalLp-blowup}
	\|T^1_j\|_{L^{p_1}(\R^2)\times L^{p_2}(\R^2)\to L^p(\R^2)}\lesssim j^3 \|\Omega\|_{L^1(\S^1)}.
        \end{equation}
	\end{lemma} 

\begin{proof}
Let $k_1$, $k_2$ be positive integers such that $\max\{k_1,k_2\}=j$. Similarly as in the proof of Proposition~\ref{P:smoothing-diagonal}, we will denote by $T$ the twisted paraproduct associated with the symbol
\begin{equation}\label{E:symbol-k12}
\sum_{\ell \in \Z} \wh{K^0}(2^\ell(\xi,\eta)) \wh{\psi}(2^{\ell-k_1}\xi) \wh{\psi}(2^{\ell-k_2}\eta).
\end{equation}
Inequality~\eqref{E:diagonalLp-blowup} will follow if we prove
\[
\|T\|_{L^{p_1}(\R^2)\times L^{p_2}(\R^2)\to L^p(\R^2)}\lesssim j^2 \|\Omega\|_{L^1(\S^1)}.
\]

The inverse Fourier transform of~\eqref{E:symbol-k12} becomes
	\begin{align*} 
		&\sum_{\ell\in \Z}  \left((K^0)_{2^{-\ell}}\ast (\psi_{2^{k_1-\ell}} \otimes \psi_{2^{k_2-\ell}}) \right)(s, t)
		=\sum_{\ell\in \Z} \int_{\R^2} K^0(s',t') \psi_{2^{k_1-\ell}}^{2^{k_1}s'}(s) \psi_{2^{k_2-\ell}}^{2^{k_2} t'}(t) ds'dt'.
	\end{align*}
We fix $(s',t') \in \R^2$ and denote by $T^{s',t'}$ the twisted paraproduct associated with the kernel
\[
\sum_{\ell\in \Z} \psi_{2^{k_1-\ell}}^{2^{k_1}s'}(s) \psi_{2^{k_2-\ell}}^{2^{k_2} t'}(t).
\]
Let $f$, $g$ be Schwartz functions on $\R^2$. By the Cauchy-Schwarz inequality, we have
\[
\left|T^{s',t'}(f,g)(x,y)\right| \leq  \left( \sum_{\ell\in \Z} \left| f\ast^{(1)} \psi_{2^{k_1-\ell}}^{2^{k_1}s'}(x,y) \right|^2 \right)^{1/2}  \left( \sum_{\ell\in \Z} \left| g\ast^{(2)} \psi_{2^{k_2-\ell}}^{2^{k_2} t'}(x,y)\right|^2\right)^{1/2}.
\]
An application of H\"older's inequality and of the estimate~\eqref{E:shifted-square-function} then yields
\begin{align*}
&\|T^{s',t'}(f,g)\|_{L^p(\R^2)} \\
&\leq \left\|\left( \sum_{\ell\in \Z} \left| f\ast^{(1)} \psi_{2^{k_1-\ell}}^{2^{k_1}s'}(x,y) \right|^2 \right)^{1/2}\right\|_{L^{p_1}(\R^2)}
\left\|\left( \sum_{\ell\in \Z} \left| g\ast^{(2)} \psi_{2^{k_2-\ell}}^{2^{k_2} t'}(x,y)\right|^2\right)^{1/2}\right\|_{L^{p_2}(\R^2)}\\
&\lesssim \log(2+2^{k_1}|s'|)\log(2+2^{k_2}|t'|) \|f\|_{L^{p_1}(\R^2)} \|g\|_{L^{p_2}(\R^2)}.
\end{align*}
By Minkowski's inequality, the $L^{p_1}(\R^2) \times L^{p_2}(\R^2) \rightarrow L^p(\R^2)$ norm of the operator $T$ is thus bounded by a multiple of
\[
j^2 \int_{\R^2} |K^0(s',t')| \log(2+|s'|)\log(2+|t'|)\,ds' dt'
\lesssim j^2 \|\Omega\|_{L^1(\S^1)}. 
\]
This completes the proof.
\end{proof}

Proposition~\ref{pr:shifted} now follows by estimating the $L^1$-norm of $\Omega$ by its $L^q$-norm for $q>1$ and then interpolating between the estimates of Lemma~\ref{L:shifted} and Proposition~\ref{P:smoothing-diagonal} via a bilinear version of the Marcinkiewicz interpolation theorem, see, e.g., \cite{GLLZ}.


\subsection{Proof of Proposition~\ref{P:shifted-decay}}\label{sub:2}
To prove Proposition~\ref{P:shifted-decay}, we will need two lemmas. The first one shows an exponential decay of a weighted $L^1$-norm of the kernel $K^{0,2}_j$ as $j\to \infty$.

\begin{lemma}\label{L:integral-estimate}
Let $1<q\leq \infty$, $0\leq L<\infty$ and let $j\geq 1$ be an integer.
Then there is a constant $\delta>0$ depending on $q$ and $L$ such that
\begin{equation}\label{E:lq}
\int_{\R^2} |K_j^{0,2}(s',t')| (1+|s'|)^L(1+|t'|)^L\,ds' dt' 
\lesssim 2^{-\delta j} \|\Omega\|_{L^q(\S^1)}.
\end{equation}
\end{lemma}

\begin{proof}
Without loss of generality, we may assume that $L\geq 1$, since the estimate for $L\in [0,1)$ follows from the one for $L=1$. Additionally, 
by symmetry it suffices to prove inequality~\eqref{E:lq} with $K_j^{0,2}$ replaced by  the kernel $\mathcal K$ whose Fourier transform is given by
\[
\wh{\mathcal K}(\xi,\eta)=\wh{K^0}(\xi,\eta) \wh{\psi}(2^{-j}\xi)\wh{\varphi}(\eta).
\]

We estimate
\begin{align}\label{E:annuli}
&\int_{\R^2} |\mathcal K(s',t')| (1+|s'|)^L (1+|t'|)^{L}\,ds' dt' \\
\nonumber
&\lesssim \int_{B(0,2)}|\mathcal K(s',t')|\,ds'dt' + \sum_{k=1}^\infty 2^{2Lk} \int_{B(0,2^{k+1})\setminus B(0,2^k)} |\mathcal K(s',t')| \,ds' dt'.
\end{align}
Using H\"older's inequality and Plancherel's identity, we get
\begin{equation}\label{E:plancherel}
\int_{B(0,2)}|\mathcal K(s',t')|\,ds'dt'
\lesssim \|\mathcal K\|_{L^2(\R^2)}
=\|\wh{\mathcal K}\|_{L^2(\R^2)}.
\end{equation}
The decay estimate~\eqref{E:decay-K0} then yields
\begin{align}
\nonumber
&\|\wh{\mathcal K}\|_{L^2(\R^2)}
=\|\wh{K^0}(\xi,\eta) \wh{\psi}(2^{-j}\xi)\wh{\varphi}(\eta)\|_{L^2(\R^2)}\\
\label{E:unit-ball}
&\lesssim 2^{-\frac{3j}{4}} \|\Omega\|_{L^\infty(\S^1)} \|\wh{\psi}(2^{-j}\xi)\wh{\varphi}(\eta)\|_{L^2(\R^2)}
\lesssim 2^{-\frac{j}{4}} \|\Omega\|_{L^\infty(\S^1)}.
\end{align}
In addition, if $k\in \N$ then applying H\"older's inequality, Plancherel's identity and~\eqref{E:unit-ball}, we get
\begin{align}\label{E:blowup-k}
\int_{B(0,2^{k+1})\setminus B(0,2^k)}|\mathcal K(s',t')|\,ds'dt'
\lesssim 2^k\|\mathcal K\|_{L^2(\R^2)}
\lesssim 2^k 2^{-\frac{j}{4}} \|\Omega\|_{L^\infty(\S^1)}.
\end{align}

We next focus on proving estimates with $\Omega \in L^1(\S^1)$. We recall that
\[
\mathcal K=K^0 \ast [\psi_{2^j} \otimes\varphi].
\]
By Young's convolution inequality, we obtain
\begin{equation}\label{E:young}
\int_{B(0,2)}|\mathcal K(s',t')|\,ds'dt' 
\lesssim \|K^0\|_{L^1(\R^2)}\|\psi_{2^j} \otimes\varphi\|_{L^1(\R^2)} \lesssim \|\Omega\|_{L^1(\S^1)}.
\end{equation}
Next, we prove that
\begin{equation}\label{E:est-l1}
\int_{B(0,2^{k+1})\setminus B(0,2^k)}|\mathcal K(s',t')|\,ds'dt' 
\lesssim \|\Omega\|_{L^1(\S^1)} 2^{-8Lk}.
\end{equation}
Since $\psi$ and $\varphi$ are Schwartz functions, we get
\[
|\psi(x)| \lesssim (1+|x|)^{-10L} 
\]
and
\[
|\varphi(y)| \lesssim (1+|y|)^{-10L}. 
\]
Thus,
\begin{align*}
|\psi_{2^j} \otimes \varphi|(x,y) 
&\lesssim 2^j (1+2^j|x|)^{-10L} (1+|y|)^{-10L}\\
&\lesssim 2^j\sum_{l=1}^\infty 2^{-10Ll} \chi_{B(0,2^l)}(2^jx,y).
\end{align*}

We observe that the function $K^0 \ast (\chi_{B(0,2^l)}(2^j\cdot,\cdot))$ is supported in the ball $B(0,2^{l+1})$, and therefore $K^0 \ast (\chi_{B(0,2^l)}(2^j\cdot,\cdot))$ vanishes on $B(0,2^{k+1})\setminus B(0,2^k)$ if $l<k-1$. This implies
\begin{align*}
\int_{B(0,2^{k+1})\setminus B(0,2^k)}|\mathcal K(s',t')|\,ds'dt'
&\lesssim 2^j \sum_{l=k-1}^\infty 2^{-10Ll} \|K^0 \ast (\chi_{B(0,2^l)}(2^j\cdot,\cdot))\|_{L^1(\R^2)}\\
&\lesssim \|\Omega\|_{L^1(\S^1)} \sum_{l=k-1}^\infty 2^{-10Ll} 2^j \|\chi_{B(0,2^l)}(2^j\cdot,\cdot)\|_{L^1(\R^2)}\\
&\lesssim \|\Omega\|_{L^1(\S^1)} \sum_{l=k-1}^\infty 2^{-8Ll}
\lesssim \|\Omega\|_{L^1(\S^1)} 2^{-8Lk}. 
\end{align*}
This yields~\eqref{E:est-l1}. 

Combining inequalities~\eqref{E:annuli},~\eqref{E:young} and~\eqref{E:est-l1}, we obtain
\begin{equation}\label{E:l1}
\int_{\R^2} |\mathcal K(s',t')| (1+|s'|)^L (1+|t'|)^{L}\,ds' dt' \lesssim \|\Omega\|_{L^1(\S^1)}.
\end{equation}
Additionally, we take the geometric mean of the estimates~\eqref{E:blowup-k} and~\eqref{E:est-l1} to deduce that
\begin{equation}\label{E:140}
\int_{B(0,2^{k+1})\setminus B(0,2^k)}|\mathcal K(s',t')|\,ds'dt' 
\lesssim \|\Omega\|_{L^\infty(\S^1)} 2^{-\frac{j}{8}} 2^{-3Lk}.
\end{equation}
Inequalities~\eqref{E:plancherel}, \eqref{E:unit-ball} and~\eqref{E:140} then yield
\begin{equation}\label{E:linfinity}
\int_{\R^2} |\mathcal K(s',t')| (1+|s'|)^L (1+|t'|)^{L}\,ds' dt' \lesssim 2^{-\frac{j}{8}} \|\Omega\|_{L^\infty(\S^1)}.
\end{equation}
We point out that inequalities~\eqref{E:l1} and~\eqref{E:linfinity} hold for all functions $\Omega \in L^1(\S^1)$ or $\Omega \in L^\infty(\S^1)$, respectively; the mean zero condition on $\Omega$ was not required in our arguments. 
To finish the proof, we view $\mathcal K$ as a linear operator acting on the function $\Omega$ and interpolate between the estimates~\eqref{E:l1} and~\eqref{E:linfinity} via the Riesz-Thorin interpolation theorem. This yields~\eqref{E:lq}.
\end{proof}

The second lemma establishes logarithmic estimates for the shifted twisted paraproduct, as announced in~\eqref{E:logarithmic-blowup}. For our applications, it is sufficient to obtain estimates which are logarithmic in $u$; this is caused by the anisotropic scaling in~\eqref{E:off-diagonal-kernel} (i.e., we do not rescale the function $\wh{\varphi}$). 

\begin{lemma}\label{L:shifted-twisted-paraproduct}
Let $1<p_1,p_2<\infty$, $0<p<\infty$ satisfy $\frac{1}{p_1}+\frac{1}{p_2}=\frac{1}{p}$, and let $L$ be a positive integer. Assume that there is a constant $C>0$ such that inequality~\eqref{E:kovac} holds for all Calder\'on-Zygmund kernels $K$ of order $L$. Let $\rho$, $\phi$ be Schwartz functions on $\R$ whose Fourier transforms are supported in $[-4,4]$ and $[-4,-1/4] \cup [1/4,4]$, respectively. For $(u,v)\in \R^2$, we define the family $K^{u,v}$ of shifted kernels as in~\eqref{E:kuv}. Then
\begin{equation}\label{E:logarithmic-blowup-2}
\|T_{K^{u,v}}\|_{L^{p_1}(\R^2) \times L^{p_2}(\R^2) \rightarrow L^p(\R^2)} \lesssim \log(2+|u|) (1+|v|)^{L}.
\end{equation}
\end{lemma}

\begin{proof}
We split the sum over $k$ in~\eqref{E:kuv} into $10$ pieces modulo $10$. Estimates for all these pieces are analogous, so we will only consider the kernel
\[
\mathcal K^{u,v}(s,t)=\sum_{k\in 10\Z} \phi_{2^{-k}}^u(s) \rho_{2^{-k}}^v(t). 
\]

We fix $k\in 10\Z$, $l\in 10\Z$ and $(u,v)\in \R^2$. Let $\widetilde{\phi}$ be a Schwartz function on $\R$ whose Fourier transform is supported in $[-8,-1/8] \cup [1/8,8]$ and equals $1$ on $[-4,-1/4] \cup [1/4,4]$. Then
\begin{equation}\label{E:convolution}
\phi_{2^{-k}}^{u} \ast \widetilde{\phi}_{2^{-k}}=\phi_{2^{-k}}^{u} 
\end{equation}
and
\begin{equation}\label{E:convolution-2}
\phi_{2^{-l}}^{u} \ast \widetilde{\phi}_{2^{-k}}=0 \quad \text{if } l\neq k.
\end{equation}
Let $f$ be a Schwartz function on $\R^2$. We define its shifted variant $f^{u}$ by
\begin{equation}\label{E:definition-f-tilde}
f^u=\sum_{l\in 10\Z} f \ast^{(1)} \phi_{2^{-l}}^u.
\end{equation}
Using~\eqref{E:convolution} and~\eqref{E:convolution-2}, we get
\begin{equation}\label{E:kl}
f^{u} \ast^{(1)} \widetilde{\phi}_{2^{-k}}
=\sum_{l\in 10 \Z} f \ast^{(1)} (\phi_{2^{-l}}^{u} \ast \widetilde{\phi}_{2^{-k}})
=f\ast^{(1)} \phi_{2^{-k}}^{u}.
\end{equation}

Let $g$ be a Schwartz function on $\R^2$. We observe that
\begin{align}\label{E:move-shift}
T_{\mathcal K^{u,v}}(f,g)(x,y)
&=\sum_{k\in 10\Z} f \ast^{(1)} \phi_{2^{-k}}^u(x,y) g\ast^{(2)} \rho_{2^{-k}}^v(x,y)\\
\nonumber
&=\sum_{k\in 10\Z} f^{u} \ast^{(1)} \widetilde{\phi}_{2^{-k}}(x,y) g\ast^{(2)} \rho_{2^{-k}}^v(x,y).
\end{align}
Let $\mathcal{K}^v$ be the kernel defined by
\[
\mathcal K^v(s,t)=\sum_{k\in 10\Z} \widetilde{\phi}_{2^{-k}}(s)\rho_{2^{-k}}^v (t).
\]
Thanks to~\eqref{E:move-shift}, $T_{\mathcal K^{u,v}}(f,g)=T_{\mathcal K^v}(f^u,g)$. 

The Fourier transform of $\mathcal K^v$ satisfies
\[
\wh{\mathcal K^v}(\xi,\eta)=\sum_{k\in 10 \Z} \wh{\widetilde{\phi}}(2^k\xi) \wh{\rho}(2^k\eta) e^{-2\pi i 2^k v\eta}.
\]
Thus, there is a constant $c$ independent of $v$ such that $c(1+|v|)^{-L}\mathcal K^v$ is a Calder\'on-Zygmund kernel of order $L$. Consequently,
\[
\|T_{\mathcal K^{u,v}}(f,g)\|_{L^p(\R^2)}
=\|T_{\mathcal K^v}(f^u,g)\|_{L^p(\R^2)}
\lesssim (1+|v|)^{L}\|f^u\|_{L^{p_1}(\R^2)} \|g\|_{L^{p_2}(\R^2)}. 
\]
It remains to observe that 
\[
\|f^u\|_{L^{p_1}(\R^2)} \lesssim \log(2+|u|)\|f\|_{L^{p_1}(\R^2)}.
\]
Indeed, for a fixed $y\in \R$, the Fourier transform of the function $x \mapsto f \ast^{(1)} \phi_{2^{-k}}^u(x,y)$ is supported in the set $2^{-k-2} \leq |\xi|\leq 2^{-k+2}$. Applying~\cite[Lemma 7.5.2(b)]{MFA} in the first variable, we get
\[
\|f^u\|_{L^{p_1}(\R^2)} \lesssim \left\|\left(\sum_{l\in 10\Z} |f \ast^{(1)} \phi_{2^{-l}}^u|^2\right)^{\frac{1}{2}}\right\|_{L^{p_1}(\R^2)}.
\]
The proof is then concluded by using the estimate~\eqref{E:shifted-square-function} in the first variable.
\end{proof}

\begin{remark}
As pointed out in the introduction, estimates for a certain quadrilinear form with cubical structure associated with the family $K^{u,v}$ of shifted kernels~\eqref{E:kuv} were established in~\cite[Lemma 2.5 and Lemma 2.6]{DS2026}.
In these estimates, the authors were able to remove the dependence of the bound on $u$ by an application of the Cauchy-Schwarz inequality; such an argument is however global in nature and does not immediately yield any conclusion for the twisted paraproduct. A suitable localization of the techniques from~\cite{DS2026} would likely be applicable in our setting, but the localization procedure is again expected to produce a logarithmic blowup of the bounds in $u$. We do not pursue this approach here as our argument appears to be simpler. On the other hand, our approach of moving the shift from the kernel to one of the input functions only works for the twisted paraproduct and is not applicable in the more general setting of quadrilinear forms. 
\end{remark}

\begin{proof}[Proof of Proposition~\ref{P:shifted-decay}]
The symbol of the operator $T_j^2$ has the form 
\[
\sum_{\ell \in \Z} \wh{K_j^{0,2}}(2^{\ell}(\xi,\eta)).
\]
The function $\wh{K_j^{0,2}}(\xi,\eta)$ splits further into two parts as in~\eqref{E:off-diagonal-kernel}. By symmetry, we only consider the first part given by
\[
\wh{\mathcal K}(\xi,\eta):=\wh{K^0}(\xi,\eta) \wh{\psi}(2^{-j}\xi)\wh{\varphi}(\eta).
\]
Let $\phi$ be a Schwartz function on $\R$ whose Fourier transform is supported in $[-4,-1/4] \cup [1/4,4]$ and is equal to $1$ on $[-2,-1/2] \cup [1/2,2]$. Additionally, let $\rho$ be a Schwartz on $\R$ whose Fourier transform is supported in $[-4,4]$ and equal to $1$ on $[-2,2]$. Then $\wh{\phi}=1$ on the support of $\wh{\psi}$ and $\wh{\rho}=1$ on the support of $\wh{\varphi}$. Therefore, our symbol can be written as
\begin{equation}\label{E:symbol-1}
\sum_{\ell\in \Z} \wh{\mathcal K}(2^\ell(\xi,\eta)) \wh{\phi}(2^{\ell-j}\xi) \wh{\rho}(2^{\ell}\eta)
=\sum_{k\in \Z} \wh{\mathcal K}(2^{k+j}(\xi,\eta)) \wh{\phi}(2^{k}\xi) \wh{\rho}(2^{k+j}\eta).
\end{equation}
As the $L^{p_1}(\R^2) \times L^{p_2}(\R^2) \rightarrow L^p(\R^2)$ norm of the twisted paraproduct is invariant under anisotropic rescaling of the symbol, we may work with the symbol
\begin{equation}\label{E:rescaled-symbol}
\sum_{k\in \Z} \wh{\mathcal K}(2^{k+j}\xi,2^k\eta) \wh{\phi}(2^{k}\xi) \wh{\rho}(2^{k}\eta)
\end{equation}
in place of~\eqref{E:symbol-1}.
The inverse Fourier transform of~\eqref{E:rescaled-symbol} becomes
	\begin{equation}\label{E:kernel}
		\sum_{k\in \Z} \int_{\R^2} \mathcal K(s',t') \phi_{2^{-k}}^{2^{j}s'}(s) \rho_{2^{-k}}^{t'}(t) ds'dt'. 
	\end{equation}

We fix $(s',t')$. Thanks to~\cite{Kovac}, there are positive constants $C$ and $L$ such that inequality~\eqref{E:kovac} holds for all Calder\'on-Zygmund kernels of order $L$.
Therefore, Lemma~\ref{L:shifted-twisted-paraproduct} yields that the operator $T_{K^{2^js',t'}}$ associated with the kernel
\[
K^{2^js',t'}(s,t)=\sum_{k\in \Z} \phi_{2^{-k}}^{2^{j}s'}(s) \rho_{2^{-k}}^{t'}(t)
\]
satisfies
\[
\|T_{K^{2^js',t'}}\|_{L^{p_1}(\R^2) \times L^{p_2}(\R^2) \rightarrow L^p(\R^2)} \lesssim \log(2+2^j|s'|) (1+|t'|)^{L}
\lesssim j (1+|s'|)^L(1+|t'|)^{L}.
\]
By Minkowski's inequality, the $L^{p_1}(\R^2) \times L^{p_2}(\R^2) \rightarrow L^p(\R^2)$ norm of the twisted paraproduct associated with the kernel~\eqref{E:kernel} is bounded by a multiple of
\[
j \int_{\R^2} |\mathcal K(s',t')| (1+|s'|)^L (1+|t'|)^L\,ds' dt'.
\]
The proof is now concluded by applying Lemma~\ref{L:integral-estimate}. 
\end{proof}

\section{Proof of Proposition~\ref{P:interpolation-quasi-banach}: fiberwise Calder\'on-Zygmund decomposition}\label{S:CZ}

Having Corollary~\ref{P:interpolation-banach} at our disposal, we proceed by performing a fiberwise Calder\'on-Zygmund decomposition to extend our estimates to the quasi-Banach setting. We follow the arguments in~\cite[proof of Lemma 5.4]{BSpreprint},~\cite[proof of Proposition 5.1]{HP}, applying them fiberwise as in Bernicot~\cite{Bernicot}. The proof below is thus rather standard, we nevertheless include it for the sake of completeness.

We recall that if $1<r<\infty$, then the Lorentz space $L^{r,\infty}(\R^2)$ consists of all measurable functions $f$ on $\R^2$ for which
\[
\|f\|_{L^{r,\infty}(\R^2)} =\sup_{t>0} t |\{x\in \R^2:~|f(x)|>t\}|^{\frac{1}{r}}. 
\]

\begin{lemma}\label{pr:fiberCZ}
Let $1<p_0<\infty$ and $p_0'<q\leq \infty$. Let $j\geq 1$ be an integer. Then 
	\begin{align}
		\label{E:off-diagonal-p<1}	\|T_j\|_{L^{1}(\R^2)\times L^{p_0}(\R^2)\to L^{\frac{p_0}{p_0+1},\infty}(\R^2)}\lesssim j \|\Omega\|_{L^q(\S^1)},\\
		\label{E:off-diagonal-p<1,2}	\|T_j\|_{L^{p_0}(\R^2)\times L^{1}(\R^2)\to L^{\frac{p_0}{p_0+1},\infty}(\R^2)}\lesssim j \|\Omega\|_{L^q(\S^1)}.	
	\end{align}
\end{lemma}

\begin{proof} 
By symmetry, it is enough to prove inequality~\eqref{E:off-diagonal-p<1}. Let $\mathcal{D}$ be the set of those functions $f\in L^1(\R^2)$ which can be written in the form 
\begin{equation}\label{E:form-of-f}
f(x,y) = \sum_j f^1_j(x)\chi_{E_j}(y),
\end{equation}
where $(f^1_j)_j$ is a finite sequence of functions from $L^1(\R)$ and $(E_j)_j$ is a finite sequence of pairwise disjoint measurable subsets of $\R$. We fix two functions $f\in\mathcal{D}$, $g\in L^{p_0}(\R^2)$ and normalize $\|f\|_{L^1(\R^2)} = \|g\|_{L^{p_0}(\R^2)} = \|\Omega\|_{L^q(\S^1)} = 1$. It is enough to show 
	\[ \big|\big\{(x,y)\in \R^2 : |T_j(f,g)(x,y)|>2\lambda \big\}\big| \lesssim j \lambda^{-\frac{p_0}{p_0+1}}, \quad \lambda >0,
	\]
	since $\mathcal D$ is dense in $L^1(\R^2)$. We fix $\lambda>0$ and perform a Calder\'on-Zygmund decomposition of $f(\cdot,y)$ at height $\lambda^{\frac{p_0}{p_0+1}}$ for each $y\in \R$. This yields a bounded function $b_y$ and mean zero functions $(a_{i,y})_i$ supported in disjoint dyadic intervals $Q_{i,y}$ such that 
	\begin{itemize}
		\item for all $(x,y)\in \R^2$, $f(x,y) = b_y(x) + \sum_i a_{i,y}(x)$,\\
		\item for all $y\in \R$, $\|b_y\|_{L^1(\R)} \lesssim \|f(\cdot,y)\|_{L^1(\R)}$ and $\|b_y\|_{L^\infty(\R)} \lesssim \lambda^{\frac{p_0}{p_0+1}}$,\\
		\item for all $i$ and all $y\in \R$, $\int_{Q_{i,y}} a_{i,y} = 0$ and $\|a_{i,y}\|_{L^1(\R)} \lesssim \lambda^{\frac{p_0}{p_0+1}} |Q_{i,y}|$,\\
		\item for all $i$ and all $y\in \R$, $\big|\bigcup_i Q_{i,y}\big| = \sum_i |Q_{i,y}| \lesssim \lambda^{-\frac{p_0}{p_0+1}}\|f(\cdot,y)\|_{L^1(\R)}$.
	\end{itemize}

	We set $b(x,y) = b_y(x)$ and observe that, thanks to~\eqref{E:form-of-f}, $b$ is measurable. Since $\|b\|_{L^1(\R^2)} \lesssim 1$ and $\|b\|_{L^\infty(\R^2)} \lesssim \lambda^{\frac{p_0}{p_0+1}}$, we deduce that
	\[\|b\|_{L^r(\R^2)} \lesssim \lambda^{\frac1{r'}\frac{p_0}{p_0+1}}, \quad 1\leq r\leq \infty.\]
To establish the weak-type bound for $b$, we use the fact that the operator $T_j$ is bounded from $L^{2p_0'}(\R^2) \times L^{p_0}(\R^2)$ into $L^{\frac{2p_0}{p_0+1}}(\R^2)$ with norm bounded by a constant independent of $j$. This follows from Corollary~\ref{P:interpolation-banach} paired with the inequalities $1<\frac{2p_0}{p_0 +1 } <2$ and $1< 2p_0' , p_0 <\infty $. Therefore, 
	\begin{align*} \big|\big\{(x,y)\in \R^2 : |T_j(b,g)(x,y)|>\lambda \big\}\big| &\lesssim \lambda^{-\frac{2p_0}{p_0 +1 }} \big\|T_j(b,g)\big\|_{L^{\frac{2p_0}{p_0 +1 }}(\R^2)}^\frac{2p_0}{p_0 +1 }	\\
	&\lesssim \lambda^{-\frac{2p_0}{p_0 +1 }}\|b\|_{L^{2p_0 '}(\R^2)}^\frac{2p_0}{p_0 +1 }\\
	&=\lambda^{-\frac{p_0}{p_0 +1 }}.
	\end{align*}

We next focus on proving the weak-type bound for $a$, where $a(x,y)=\sum_{i} a_{i,y}(x)$. For any $i$ and $y\in \R$, we set $Q^*_{i,y}$ to be the interval with the same center as $Q_{i,y}$ and four-times the length. We also denote $\mathcal{A}^*=\big\{(x,y)\in \R^2 : x\in \bigcup_i Q^*_{i,y}\}$.
%
	%
	%
Since 
	\begin{align*}\big|\mathcal{A}^*\big|
		\lesssim \lambda^{-\frac{p_0}{p_0 +1 }} \int_{\R} \|f(\cdot,y)\|_{L^1(\R)} dy
		=\lambda^{-\frac{p_0}{p_0 +1 }},
	\end{align*}
	it is enough to prove 
	\[ \Big|\Big\{(x,y)\notin\mathcal{A}^* : |T_j(a,g)(x,y)|>\lambda \Big\}\Big| \lesssim j \lambda^{-\frac{p_0}{p_0+1}}.
	\]

Using~\eqref{E:kj} and~\eqref{E:varphi}, we see that
\[
\wh{K_j}(\xi,\eta)
=\sum_{\ell \in \Z} \wh{K^0}(2^\ell(\xi,\eta)) \wh{\psi}(2^{\ell-j}\xi) \wh{\varphi}(2^{\ell-j}\eta)
+\sum_{\ell \in \Z} \wh{K^0}(2^\ell(\xi,\eta)) \wh{\varphi}(2^{\ell-j+1}\xi) \wh{\psi}(2^{\ell-j}\eta).
\]
Since both terms on the right-hand side are symmetric, we will only focus on the first one. The associated operator is then given by
\begin{equation}\label{E:bad-part}
\widetilde{T}_j(a,g)(x,y)
=\sum_{\ell \in \Z} 2^{-2\ell} \int_{\R^2} K^0(2^{-\ell}s,2^{-\ell}t) a \ast^{(1)} \psi_{2^{j-\ell}}(x-s,y) g \ast^{(2)} \varphi_{2^{j-\ell}}(x,y-t)\,ds dt.
\end{equation}
We split
\[
a(x,y)=\sum_{k\in \Z} a^k(x,y),
\]
where $a^k(x,y)$ is the sum of those functions $a_{i,y}(x)$ for which the length of the interval $Q_{i,y}$ is $2^{-k}$. Then~\eqref{E:bad-part} becomes
\begin{equation}\label{E:operator}
\sum_{\ell\in \Z} \sum_{k\in \Z} \int_{\R^2} K^0(s,t) a^k \ast^{(1)} \psi_{2^{j-\ell}}(x-2^{\ell}s,y) g \ast^{(2)} \varphi_{2^{j-\ell}}(x,y-2^{\ell}t)\,ds dt.
\end{equation}
For $k\in \Z$, we set
\[
\psi_{2^{j-\ell},k}=
\begin{cases}
\psi_{2^{j-\ell}}, &k \geq -\ell,\\
\chi_{(-\infty,-\cdot2^{-k-2})\cup(2^{-k-2},\infty)}\psi_{2^{j-\ell}}, &k < -\ell
\end{cases}
\]
and observe that for $(x,y) \notin \mathcal A^*$ and $|s|\leq 2$, we have $a^k \ast^{(1)} \psi_{2^{j-\ell}}(x-2^{\ell}s,y)=a^k \ast^{(1)} \psi_{2^{j-\ell},k}(x-2^{\ell}s,y)$. We may thus estimate the absolute value of~\eqref{E:operator} by a multiple of
\[
\int_{\R^2} |K^0(s,t)| \sum_{\ell \in \Z} \sum_{k\in \Z} |a^k \ast^{(1)} \psi_{2^{j-\ell},k}(x-2^{\ell}s,y)| \mathcal M^{(2)} g(x,y-2^\ell t)\,ds dt,
\]
where $\mathcal M^{(2)}$ denotes the Hardy-Littlewood maximal operator in the second variable. Using the support properties of $K^0$ and H\"older's inequality, this is bounded further by
\begin{align*}
&\int_{-2}^2 \sum_{\ell \in \Z} \sum_{k\in \Z} |a^k \ast^{(1)} \psi_{2^{j-\ell},k}(x-2^{\ell}s,y)| \int_{-2}^2 |K^0(s,t)| \mathcal M^{(2)} g(x,y-2^\ell t)\,dt ds\\
&\leq \int_{-2}^2 \sum_{\ell \in \Z} \sum_{k\in \Z} |a^k \ast^{(1)} \psi_{2^{j-\ell},k}(x-2^{\ell}s,y)| \left(\int_{-2}^2 |K^0(s,t)|^q \,dt\right)^{\frac{1}{q}}\\
&\quad \quad \quad \quad \times \left(\int_{-2}^2 |\mathcal M^{(2)} g(x,y-2^\ell t)|^{q'}\,dt\right)^{\frac{1}{q'}}\,ds\\
&\lesssim \mathcal M_{q'}^{(2)} \mathcal M^{(2)} g(x,y) \int_{-2}^2 \sum_{\ell \in \Z} \sum_{k\in \Z} |a^k \ast^{(1)} \psi_{2^{j-\ell},k}(x-2^{\ell}s,y)| \left(\int_{-2}^2 |K^0(s,t)|^q \,dt\right)^{\frac{1}{q}}\,ds,
\end{align*}
where $\mathcal M^{(2)}_{q'} h=(\mathcal M^{(2)} |h|^{q'})^{\frac{1}{q'}}$.
Thus,
\begin{align}\label{E:two-terms}
&|\{(x,y)\notin \mathcal A^*:~|T_j(a,g)(x,y)|>\lambda\}|\\
\nonumber
&\lesssim |\{(x,y)\in \mathbb R^2: \mathcal M_{q'}^{(2)} \mathcal M^{(2)} g(x,y)>\lambda^{\frac{1}{1+p_0}}\}|+|\{(x,y)\notin \mathcal A^*:\\
\nonumber
&\quad \quad \quad \quad\int_{-2}^2 \sum_{\ell \in \Z} \sum_{k\in \Z} |a^k \ast^{(1)} \psi_{2^{j-\ell},k}(x-2^{\ell}s,y)| \left(\int_{-2}^2 |K^0(s,t)|^q \,dt\right)^{\frac{1}{q}}\,ds > \lambda^{\frac{p_0}{1+p_0}}\}|.
\end{align}
Since $q>p_0'$, the operator $\mathcal M_{q'}^{(2)} \mathcal M^{(2)}$ is bounded on $L^{p_0}(\R^2)$. In particular, it is of weak type $(p_0,p_0)$, which implies that the first term on the right-hand side of~\eqref{E:two-terms} is bounded by a multiple of $\lambda^{-\frac{p_0}{1+p_0}}$. 

We dominate the second term on the right-hand side of~\eqref{E:two-terms} by
\begin{align}\label{E:chebyshev}
\nonumber
&\lambda^{-\frac{p_0}{1+p_0}} \int_{\R^2} \int_{-2}^2 \sum_{\ell \in \Z} \sum_{k\in \Z} |a^k \ast^{(1)} \psi_{2^{j-\ell},k}(x-2^{\ell}s,y)| \left(\int_{-2}^2 |K^0(s,t)|^q \,dt\right)^{\frac{1}{q}}\,ds dx dy\\
\nonumber
&=\lambda^{-\frac{p_0}{1+p_0}}\int_{-2}^2 \left(\int_{-2}^2 |K^0(s,t)|^q \,dt\right)^{\frac{1}{q}}\,ds \sum_{\ell \in \Z} \sum_{k\in \Z} \|a^k \ast^{(1)} \psi_{2^{j-\ell},k}\|_{L^1(\R^2)}\\
&\lesssim\lambda^{-\frac{p_0}{1+p_0}} \|\Omega\|_{L^q(\S^1)} \sum_{k \in \Z} \sum_{\ell\in \Z} \|a^k \ast^{(1)} \psi_{2^{j-\ell},k}\|_{L^1(\R^2)}. 
\end{align}
Applying~\cite[Lemma 5.7]{BSpreprint} in the first variable, we obtain
\[
\|a^k \ast^{(1)} \psi_{2^{j-\ell},k}\|_{L^1(\R^2)} \lesssim \min\{1,2^{j-\ell-k}\} \|a^k\|_{L^1(\R^2)} \quad \text{if } k \geq -\ell
\]
and
\[
\|a^k \ast^{(1)} \psi_{2^{j-\ell},k}\|_{L^1(\R^2)} \lesssim 2^{-(j-\ell-k)} \|a^k\|_{L^1(\R^2)} \quad \text{if } k < -\ell.
\]
Using these estimates and summing in $\ell$, we bound~\eqref{E:chebyshev} by
\[
j \lambda^{-\frac{p_0}{1+p_0}}\|\Omega\|_{L^q(\S^1)} \sum_{k \in \Z} \|a^k\|_{L^1(\R^2)}
\lesssim j \lambda^{-\frac{p_0}{1+p_0}},
\]
as desired.
\end{proof}

\begin{proof}[Proof of Proposition~\ref{P:interpolation-quasi-banach}]
By choosing $\varepsilon\in (0,1/4)$ to be sufficiently small, we ensure that 
\[
0<\frac{1}{p_1}-\frac{1+\varepsilon}{2p}+\frac{1}{2}<1, \quad 0<\frac{1}{p_2}-\frac{1+\varepsilon}{2p}+\frac{1}{2}<1.
\]
We then set
\[
\frac{1}{\widetilde{p}_1}:=\frac{1}{p_1}-\frac{1+\varepsilon}{2p}+\frac{1}{2}, 
\quad \frac{1}{\widetilde{p}_2}:=\frac{1}{p_2}-\frac{1+\varepsilon}{2p}+\frac{1}{2},
\quad \frac{1}{\widetilde{p}}:=\frac{1}{\widetilde{p}_1}+\frac{1}{\widetilde{p}_2}
\]
and observe that $1<\widetilde{p}<2$. Corollary~\ref{P:interpolation-banach} thus yields that there is $\delta>0$ such that
\[
\|T_j\|_{L^{\widetilde{p}_1}(\R^2)\times L^{\widetilde{p}_2}(\R^2) \rightarrow L^{\widetilde{p}}(\R^2)} \lesssim 2^{-j\delta} \|\Omega\|_{L^q(\S^1)}. 
\]
Additionally, since $1/p+1/q<2$, for a sufficiently small $\varepsilon>0$ we have 
\[
1<\frac{p}{1-p+\varepsilon}<\infty, \quad q>\left(\frac{p}{1-p+\varepsilon}\right)^{'}.
\]
It then follows from Lemma~\ref{pr:fiberCZ} that
\[
\|T_j\|_{L^1(\R^2) \times L^{\frac{p}{1-p+\varepsilon}}(\R^2) \rightarrow L^{\frac{p}{1+\varepsilon},\infty}(\R^2)} \lesssim j \|\Omega\|_{L^q(\S^1)},
\]
\[
\|T_j\|_{L^{\frac{p}{1-p+\varepsilon}}(\R^2)\times L^1(\R^2) \rightarrow L^{\frac{p}{1+\varepsilon},\infty}(\R^2)} \lesssim j \|\Omega\|_{L^q(\S^1)}.
\]
Finally, by lowering our $\varepsilon>0$ further, if necessary, we ensure that the point $(1/p_1,1/p_2)$ lies inside the triangle with vertices
\[
\left(\frac{1}{\widetilde{p}_1}, \frac{1}{\widetilde{p}_2}\right), \quad
\left(1,\frac{1+\varepsilon-p}{p}\right), \quad \left(\frac{1+\varepsilon-p}{p},1\right). 
\]
Applying the bilinear version of the Marcinkiewicz interpolation theorem~\cite{GLLZ}, we obtain~\eqref{E:decay-quasibanach}. 
\end{proof}

\section{Proof of Proposition~\ref{P:counterexample}: a counterexample}\label{S:counterexample}

In this section, we construct an odd function $\Omega$ on $\S^1$ confirming the validity of Proposition~\ref{P:counterexample}.

\begin{proof}[Proof of Proposition~\ref{P:counterexample}]
Let $\Omega$ be the odd function on $\S^1$ given by
\begin{equation}\label{E:omega}
\Omega(\theta_1,\theta_2)=|\theta_1-\theta_2|^{-\frac{1}{q}} \left(\log \frac{1}{|\theta_1-\theta_2|}\right)^{-\frac{\alpha}{q}} \chi_{(-\frac{1}{2},\frac{1}{2})} (\theta_1-\theta_2) \operatorname{sgn}(\theta_1), \quad (\theta_1,\theta_2)\in \S^1, 
\end{equation}
where $\alpha>1$. We easily verify that $\Omega \in L^q(\S^1)$. 

We first show that the operator $T_\Omega$ is unbounded from $L^{p_1}(\R^2) \times L^{p_2}(\R^2)$ into $L^p(\R^2)$ if $\frac{1}{p}+\frac{1}{q}>2$. Given a positive integer $N$ greater than $10$, we define the test functions
\begin{equation}\label{E:fg}
f_N(u,v)=\chi_{(0,\frac{1}{2})}(u+v) \chi_{(-N,N)}(v), \quad
g_N(u,v)=\chi_{(0,\frac{1}{2})}(u+v) \chi_{(-N,N)}(u).
\end{equation}
We observe that
\begin{equation}\label{E:omega-f-g}
\|f_N\|_{L^{p_1}(\R^2)} \approx N^{\frac{1}{p_1}},
\quad \|g_N\|_{L^{p_2}(\R^2)} \approx N^{\frac{1}{p_2}}. 
\end{equation}
For $x\in (\frac{N}{2},N)$, $y\in (\frac{N}{2},N)$, we have
\begin{equation}\label{E:tfg}
T_\Omega(f_N,g_N)(x,y)
=\int_{\R^2} f_N(x-s,y) g_N(x,y-t) \frac{\Omega((s,t)/|(s,t)|)}{|(s,t)|^2}\,ds dt.
\end{equation}
We note that no principal value is needed here as the integrand is nonnegative. Indeed, using the support properties of $f_N$, we see that the integration in $s$ is restricted to the set where $s>x+y-1/2>0$, and the function $\Omega((s,t)/|(s,t)|)$ is nonnegative for $s>0$. Therefore, we estimate $T_\Omega(f_N,g_N)(x,y)$ from below by a multiple of
\begin{align*}
&\int_{x+y-\frac{1}{2}}^{x+y} \int_{x+y-\frac{1}{2}}^{x+y}  N^{\frac{1}{q}-2} \left(\log \frac{N}{|s-t|}\right)^{-\frac{\alpha}{q}}\,ds dt \\
&\gtrsim \int_{x+y-\frac{1}{2}}^{x+y-\frac{3}{8}} \int_{x+y-\frac{1}{4}}^{x+y}  N^{\frac{1}{q}-2} \left(\log N\right)^{-\frac{\alpha}{q}}\,ds dt
\approx N^{\frac{1}{q}-2} (\log N)^{-\frac{\alpha}{q}}.
\end{align*}
Consequently,
\begin{equation}\label{E:TN}
\|T_\Omega(f_N,g_N)\|_{L^p(\R^2)} \gtrsim N^{\frac{2}{p}+\frac{1}{q}-2} (\log N)^{-\frac{\alpha}{q}}. 
\end{equation}
A combination of~\eqref{E:omega-f-g} and~\eqref{E:TN} shows the failure of the $L^{p_1}(\R^2) \times L^{p_2}(\R^2) \rightarrow L^p(\R^2)$ boundedness of $T_\Omega$ when $\frac{1}{p}+\frac{1}{q}>2$. 

We next construct a more delicate example to show that $T_\Omega$ is unbounded from $L^{p_1}(\R^2) \times L^{p_2}(\R^2)$ into $L^p(\R^2)$ also in the limiting case $\frac{1}{p}+\frac{1}{q} = 2$.  We will make use of the same function $\Omega$ as before and of the sequence of test functions
\begin{equation}\label{E:fn}
f_N(u,v)=|u+v|^{-\frac{1}{p_1}} \left(\log \frac{1}{|u+v|}\right)^{-\frac{\alpha}{p_1}} \chi_{(0,\frac{1}{2})}(u+v) \chi_{(-N,N)}(v),
\end{equation}
\begin{equation}\label{E:gn}
g_N(u,v)=|u+v|^{-\frac{1}{p_2}} \left(\log \frac{1}{|u+v|}\right)^{-\frac{\alpha}{p_2}} \chi_{(0,\frac{1}{2})}(u+v) \chi_{(-N,N)}(u).
\end{equation}
Here, again, $\alpha>1$ and $N$ is a positive integer greater than $10$. 
We easily verify that
\[
\|f_N\|_{L^{p_1}(\R^2)} \approx N^{\frac{1}{p_1}}, \quad
\|g_N\|_{L^{p_2}(\R^2)} \approx N^{\frac{1}{p_2}}.
\]

Assume that $x \in (-N,N)$, $y\in (-N,N)$ are such that $x+y>4$. As previously, we have the explicit expression~\eqref{E:tfg} and we observe that the integrand in~\eqref{E:tfg} is nonnegative. 
Using the change of variables $u=x+y-s$, $v=x+y-t$, the integral in~\eqref{E:tfg} becomes
\begin{equation}\label{E:integral}
\int_{\R^2} f_N(u-y,y) g_N(x,v-x) \frac{\Omega((x+y-u,x+y-v)/|(x+y-u,x+y-v)|)}{|(x+y-u,x+y-v)|^2}\,du dv.
\end{equation}
We use the support properties of the functions $f_N$ and $g_N$ to observe that the integration in~\eqref{E:integral} is restricted to the set where $u\in (0,1/2)$ and $v\in (0,1/2)$. In this region, we have
\[
|x+y-u| \approx |x+y-v| \approx x+y,
\]
and thus~\eqref{E:integral} is bounded from below by a multiple of 
\begin{equation}\label{E:log}
(x+y)^{\frac{1}{q}-2}\int_0^{\frac{1}{2}} \int_0^{\frac{1}{2}} u^{-\frac{1}{p_1}} \left(\log \frac{1}{u}\right)^{-\frac{\alpha}{p_1}}
v^{-\frac{1}{p_2}} \left(\log \frac{1}{v}\right)^{-\frac{\alpha}{p_2}}
|u-v|^{-\frac{1}{q}} \left(\log \frac{x+y}{|u-v|}\right)^{-\frac{\alpha}{q}}\,du dv.
\end{equation}

We set 
\[
A=\{(u,v)\in (0,1/2)^2:~2u \leq v \leq 4u\}. 
\]
We bound the integral in~\eqref{E:log} from below by the integral over the smaller region $A$ and observe that 
if $(u,v)\in A$ then $u \approx v \approx |u-v|$. Then~\eqref{E:log} is bounded from below by a multiple of
\begin{align*}
&(x+y)^{\frac{1}{q}-2}\int_0^{\frac{1}{8}} \int_{2u}^{4u} u^{-\frac{1}{p}-\frac{1}{q}} \left(\log \frac{x+y}{u}\right)^{-\frac{\alpha}{p}-\frac{\alpha}{q}}\,dv du\\
&\gtrsim (x+y)^{-\frac{1}{p}} \int_0^{\frac{1}{8}} u^{-1} \left(\log \frac{x+y}{u}\right)^{-2\alpha}\,du
\gtrsim (x+y)^{-\frac{1}{p}} (\log(x+y))^{1-2\alpha}. 
\end{align*}
This yields
\[
\|T_\Omega(f_N,g_N)\|_{L^p(\R^2)} \gtrsim N^{\frac{1}{p}} (\log N)^{\frac{1}{p}+1-2\alpha}. 
\]
Choosing $\alpha>1$ such that $\frac{1}{p}+1-2\alpha>0$ contradicts the $L^{p_1}(\R^2) \times L^{p_2}(\R^2) \rightarrow L^p(\R^2)$ boundedness of $T_\Omega$. 
\end{proof}

\begin{remark}
The function $\Omega$ constructed above is supported in a neighborhood of the diagonal $\theta_1=\theta_2$ on $\S^1$. This neighborhood can be made arbitrarily small by replacing the characteristic function of $(-1/2,1/2)$ in~\eqref{E:omega} by the characteristic function of an interval $(-\gamma,\gamma)$ for some small $\gamma>0$, and then replacing the characteristic function of $(0,1/2)$ in~\eqref{E:fg},~\eqref{E:fn} and~\eqref{E:gn} by the characteristic function of $(0,\gamma)$. 

Additionally, we can construct a similar counterexample on a neighborhood of any line $\theta_1=\beta \theta_2$ on $\S^1$, where $\beta \in \R \setminus \{0\}$. We only discuss the (most interesting) limiting case $1/p+1/q=2$. In this situation, the counterexample is constructed for small $\gamma>0$, by setting 
\begin{equation}\label{E:omega-2}
\Omega(\theta_1,\theta_2)=|\theta_1-\beta \theta_2|^{-\frac{1}{q}} \left(\log \frac{1}{|\theta_1-\beta \theta_2|}\right)^{-\frac{\alpha}{q}} \chi_{(-\gamma,\gamma)} (\theta_1-\beta \theta_2) \operatorname{sgn}(\theta_1),
\end{equation}
\begin{equation}\label{E:fn-2}
f_N(u,v)=|u+\beta v|^{-\frac{1}{p_1}} \left(\log \frac{1}{|u+\beta v|}\right)^{-\frac{\alpha}{p_1}} \chi_{(0,\gamma)}(u+\beta v) \chi_{(-N,N)}(\beta v),
\end{equation}
\begin{equation}\label{E:gn-2}
g_N(u,v)=|u+\beta v|^{-\frac{1}{p_2}} \left(\log \frac{1}{|u+\beta v|}\right)^{-\frac{\alpha}{p_2}} \chi_{(0,\gamma)}(u+\beta v) \chi_{(-N,N)}(u),
\end{equation}
where $\alpha>1$ and $N \geq 10$ is an integer. Indeed,
using the change of variables $u=x+\beta y-s$, $v=x+\beta y-\beta t$, the integral in~\eqref{E:tfg} defining $T_\Omega(f_N,g_N)(x,y)$ becomes (a multiple of)
\begin{align*}
\int_{\R^2} &f_N(u-\beta y,y) g_N(x,\beta^{-1}(v-x)) \\
&\times \frac{\Omega((x+\beta y-u,\beta^{-1}(x+\beta y-v))/|(x+\beta y-u,\beta^{-1}(x+\beta y-v))|)}{|(x+\beta y-u,\beta^{-1}(x+\beta y-v))|^2}\,du dv.
\end{align*}
Assuming that $x\in (-N,N)$, $y\in (-N/|\beta|, N/|\beta|)$, $x+\beta y>4/\gamma$, this is bounded from below by a multiple of 
\[
(x+\beta y)^{-\frac{1}{p}} (\log(x+\beta y))^{1-2\alpha},
\]
and the counterexample follows analogously as in the case $\beta =1$. 

The fact that counterexamples can be constructed away from the diagonal $\theta_1=\theta_2$ is in sharp contrast with the case of the one-dimensional Coifman-Meyer operators 
\begin{equation}\label{E:coifman-meyer}
\widetilde{T}_\Omega(f,g)(x)=\int_{\R^2} f(x-s) g(x-t) \frac{\Omega((s,t)/|(s,t)|)}{|(s,t)|^2}\,ds dt.
\end{equation}
The operator~\eqref{E:coifman-meyer}
is bounded in the full quasi-Banach range of exponents $1<p_1,p_2<\infty$, $0<p<\infty$, $\frac{1}{p}=\frac{1}{p_1}+\frac{1}{p_2}$ whenever $\Omega$ is a mean zero function belonging to $L^q(\S^1)$ for some $q>1$ and supported away from the diagonal $\theta_1=\theta_2$, see~\cite{HLS}. 
\end{remark}


\end{document}